\documentclass[onefignum,onetabnum]{siamart171218}

\usepackage{geometry}                % See geometry.pdf to learn the layout options. There are
\usepackage{graphicx}
\usepackage{amssymb}
\usepackage{epstopdf}
\usepackage{url}
\usepackage{mathtools}
\usepackage{dsfont,color}
\usepackage{hyperref}
\hypersetup{
    colorlinks=true, % make the links colored
    linkcolor=blue, % color TOC links in blue
    urlcolor=red, % color URLs in red
    linktoc=all % 'all' will create links for everything in the TOC
}

 \usepackage{graphicx,epstopdf,cases,multirow,soul} % <- Preambl
 \usepackage[caption=false]{subfig} % <- Preamble
 \usepackage{array} % <- Preamble 
 \usepackage[caption=false]{subfig} % <- Preamble
\newcolumntype{R}{>{$}r<{$}} % 
\newcolumntype{V}[1]{>{[\;}*{#1}{R@{\;\;}}R<{\;]}} %
\newtheorem{prop}{Proposition}
\newtheorem{remark}{Remark}
\usepackage{algorithmic}
\Crefname{ALC@unique}{Line}{Lines} 
\usepackage[algo2e]{algorithm2e}

\usepackage{color,xcolor}

\definecolor{forestgreen(traditional)}{rgb}{0.0, 0.27, 0.13}

\definecolor{darkblue}{rgb}{0.0, 0.0, 0.55}

\DeclareMathOperator{\Fcal}{\mathcal{F}}

\DeclareMathOperator{\Rbb}{\mathbb{R}}

 \newcommand{\h}[1]{\mathbf{#1}}
 
\headers{A Scale Invariant Approach for Sparse Signal Recovery}{Y. Rahimi, C. Wang, H. Dong, Y. Lou}

\title{A Scale Invariant Approach for Sparse Signal Recovery\thanks{Submitted to the journal's Methods and Algorithms for Scientific Computing section December 18, 2018. \funding{YR and YL were partially supported by NSF grants DMS-1522786 and CAREER 1846690.} }}
\author{Yaghoub Rahimi\thanks{Department of Mathematical Sciences, University of Texas at Dallas, Richardson, TX 75080 (\email{yxr160430@utdallas.edu}, \email{chaowang.hk@gmail.com}, \email{yifei.lou@utdallas.edu}).}
\and Chao Wang\footnotemark[2] \thanks{the corresponding author.}
\and Hongbo Dong\thanks{Department of Mathematics and Statistics, Washington State University, Pullman, WA 99164 (\email{hongbo.dong@wsu.edu}).}
\and Yifei Lou\footnotemark[2]}

\begin{document}

\maketitle

% REQUIRED
\begin{abstract}
In this paper, we study the ratio of the $L_1 $ and $L_2 $ norms, denoted as $L_1/L_2$, to promote sparsity.  Due to the  non-convexity and non-linearity, there has been little attention to this scale-invariant model. Compared to popular models in the literature such as the $L_p$ model for $p\in(0,1)$ and the transformed $L_1$ (TL1), this ratio model is parameter free. Theoretically, we present a strong null space property (sNSP) and prove that any sparse vector is a local minimizer of the $L_1 /L_2 $ model provided with this sNSP condition. Computationally, we focus on a  constrained formulation that can be solved via the alternating direction method of multipliers (ADMM).  Experiments show that the proposed approach is comparable to the state-of-the-art methods in sparse recovery. In addition, a variant of the $L_1/L_2$ model to apply on the gradient is also discussed with a proof-of-concept example of the MRI reconstruction. 
\end{abstract}

% REQUIRED

% REQUIRED

\section{Introduction}
Sparse signal recovery is to find the sparsest solution of $ A\mathbf{x} = \mathbf{b}$ where $ \mathbf{x} \in \mathbb{R}^n, \mathbf{b} \in \mathbb{R}^m,$ and $ A \in \mathbb{R}^{m \times n} $ for $ m \ll n $.
This problem is often referred to as \textit{compressed sensing} (CS) in the sense that the sparse signal $\h x$ is compressible. Mathematically, this fundamental problem in CS can be  formulated  as
\begin{equation}\label{eq:l0}
\min\limits_{\mathbf{x}\in \mathbb{R}^n} \|\mathbf{x}\|_0 \quad \text{s.t.} \quad A\mathbf{x} = \mathbf{b},
\end{equation}
where $\|\h x\|_0$ is the number of nonzero entries in $\h x$. 
Unfortunately, \eqref{eq:l0} is  NP-hard \cite{natarajan95} to solve. A popular approach in CS is to replace $L_0$ by the convex  $L_1$ norm,  i.e.,
\begin{equation}\label{eq:l1}
\min\limits_{\mathbf{x}\in \mathbb{R}^n} \|\mathbf{x}\|_1 \quad \text{s.t.} \quad A\mathbf{x} = \mathbf{b}.
\end{equation}
Computationally, there are various $L_1$ minimization algorithms such as primal dual \cite{chambolleP11}, forward-backward splitting \cite{raguet2013generalized}, and alternating direction method of multipliers (ADMM) \cite{boydPCPE11admm}. 

	A major breakthrough in CS was the \textit{restricted isometry property} (RIP) \cite{CRT}, which provides a sufficient condition of minimizing the $L_1$ norm to recover the sparse signal. There is a necessary and sufficient condition given in terms of null space of the matrix $A$, thus referred to as \textit{null space property} (NSP); see \cref{def:NSP}.
	 \begin{definition}[null space property \cite{cohen2009compressed}]\label{def:NSP}
		For any matrix $A \in \mathbb{R}^{m \times n} $, we say the matrix $A$ satisfies  a null space property  (NSP)  of order $s$ if 
		%	\begin{equation}
		%	\ker(A) \backslash \{ 0 \} \subset \left\lbrace\h  v  \in \mathbb{R}^n \, \middle| \, \|\h v_S\|_1 < \|\h v_{\bar{S}}\|_1, \forall S \subset [n], |S| \leq s  \right\rbrace.
		%	\end{equation}
		\begin{equation}\label{eq:NSP}
		\left\|\h v_S\right\|_1 < \left\|\h v_{\bar{S}}\right\|_1, \  \h v \in \ker(A) \backslash \{ \h 0 \}, \ \forall S \subset [n], \ |S| \leq s,
		\end{equation}
		where $[n]:=\{1, \dots, n\}$, $\bar{S}$ is the complement of $S$, i.e.,  $ [n]\backslash S $,   and  $\h x_S$ is defined as 
			\begin{equation*}
			(\h x_S)_i = \begin{cases}
			x_i & \text{ if } i \in S,\\
			0 & \text{ otherwise}. 
			\end{cases}
			\end{equation*}
			The null space of $A$ is denoted by $\ker(A) := \{\h x \ | \ A \h x = \h 0 \}$.
	\end{definition}

	Donoho and Huo \cite{donoho2001uncertainty} proved that every $s$-sparse signal $\h x\in \mathbb {R} ^{n}$  is the unique solution to the $L_1$ minimization \eqref{eq:l1} if and only if $A$ satisfies the NSP of order $s$.  NSP quantifies the notion that vectors in the null space of $A$ should not be too concentrated on a small subset of indices. Since it is a necessary and sufficient condition, NSP is widely used in proving other exact recovery guarantees. Note that NSP is no longer necessary if ``every $s$-sparse vector'' is relaxed. A weaker\footnote{The sufficient condition of \eqref{eq:NSPratio} is weaker than the one in \eqref{eq:NSP}.} sufficient condition  for the exact $L_1$ recovery was proved by Zhang \cite{zhang2013theory}. It is stated that if a vector $\h x^*$ satisfies $A\h x^*=\h b$ and
	\begin{equation}\label{eq:NSPratio}
	\sqrt{\|\h x^*\|_0}<\frac 1 2\min\limits_{\mathbf{v}}\left\{\dfrac{\|\h v\|_1}{\|\h v\|_2}: \h v\in\mbox{ker}(A)\backslash\{\h 0\} \right\},
	\end{equation}
	then $\h x^*$ is the unique solution to both \eqref{eq:l0} and \eqref{eq:l1}.
	Unfortunately, neither RIP nor NSP can be numerically verified for a given matrix
	\cite{bandeira2013certifying,tillmann2014computational}.
	
	 Alternatively, a computable condition for $L_1$'s exact recovery is based on \textit{coherence}, which is defined as 
	 \begin{equation}
	 \mu(A) := \max_{i\neq j} \dfrac{|\h a_i^T\h a_j|}{\|\h a_i\|\|\h a_j\|},	 \end{equation}
	 for a matrix $A=[\h a_1, \dots, \h a_N].$ 
Donoho-Elad \cite{donohoE03} and Gribonval \cite{gribonval2003sparse} proved independently that 
if  $\h x^\ast$ satisfies $A\h x^*=\h b$ and 
\begin{equation}\label{eq:coherence}
\|\h x^\ast \|_0<\frac 1 2 \left(1+\frac 2 {\mu(A)}\right),
\end{equation}
then $\h x^\ast$ is the optimal solution to both \eqref{eq:l0} and  \eqref{eq:l1}. 
Clearly, the coherence  $\mu(A)$ is bounded by $[0,1]$. The inequality \eqref{eq:coherence} implies that $L_1$ may not perform well
for highly coherent matrices,  i.e., $\mu(A)\sim 1$, as $\|\h x\|_0$ is then at most one, which seldom occurs simultaneously with $A\h x^*=\h b$.

Other than the popular   $L_1$ norm, 
there are a variety of regularization functionals to promote sparsity, such as $L_p$ \cite{chartrand07,xuCXZ12,laiXY13}, $L_1$-$L_2$ \cite{yinEX14,louYHX14}, capped $L_1$  (CL1) \cite{zhang2009multi,shen2012likelihood}, and transformed $L_1$ (TL1) \cite{lv2009unified,zhangX17,zhangX18}. Most of these 
models are nonconvex, leading to difficulties in proving exact recovery guarantees and algorithmic convergence, but they tend to give better empirical results compared to the convex $L_1$ approach. For example,  it was reported in \cite{yinEX14,louYHX14} that $L_p$ gives superior results for incoherent matrices (i.e., $\mu(A)$ is small), while  $L_1$-$L_2$ is the best for the coherent scenario. In addition, TL1 is always the second best no matter whether the matrix is coherent or not \cite{zhangX17,zhangX18}.

In this paper, we study the ratio of $L_1 $ and $L_2 $ as a scale-invariant  model to approximate the desired $L_0$, which is scale-invariant itself. In  one dimensional (1D) case (i.e., $n =1$), the $L_1/L_2$ model is exactly the same as the $L_0$ model if we use the convention  $\frac{0}{0}=0$.   
The ratio of $L_1 $ and $L_2 $ was first proposed by Hoyer \cite{hoyer2002} as a sparseness measure and later highlighted in \cite{hurleyR09} as a scale-invariant model. However, there has been little attention on it due to its computational difficulties arisen from being non-convex and non-linear. There are some theorems that establish the equivalence between the $L_1/L_2 $ and the $L_0$ models, but  only restricted to nonnegative signals \cite{esserLX13,yinEX14}. We aim to apply this ratio model to arbitrary signals. On the other hand, the $L_1/L_2$ minimization has an intrinsic drawback that it tends to produce one erroneously large coefficient while suppressing the other non-zero elements, under which case the ratio is reduced. To compensate for this drawback, it is helpful to incorporate a box constraint, which will also be addressed in this paper.

Now we turn  to a sparsity-related assumption that  signal is sparse after a given transform, as opposed to signal itself being sparse. This assumption is widely used in image processing. 
For example,  a natural image, denoted by $u$, is mostly sparse after taking gradient, and hence it is reasonable to minimize the $L_0$ norm of the gradient, i.e., $\|\nabla u\|_0$. To bypass the NP-hard $L_0$ norm, 
the convex relaxation replaces $L_0$ by $L_1$, 
where the  $L_1$ norm of the gradient is the well-known total variation (TV) \cite{rudinOF92} of an image.
 A weighted $L_1$-$\alpha L_2$ model (for $\alpha>0$) on the gradient was proposed in \cite{louZOX15}, which suggested that $\alpha=0.5$ yields better results than $\alpha =1$ for image denoising, deblurring, and MRI reconstruction. The ratio of $L_1$ and $L_2$ on the image gradient was used in deconvolution and blind deconvolution
\cite{krishnan2011blind,repetti2015euclid}. We further adapt the proposed ratio model from sparse signal recovery to imaging applications, specifically focusing on MRI reconstruction.

The rest of the paper is organized as follows. 
%In \Cref{sect:review}, we give a brief review of related works in CS including models and corresponding algorithms. 
\Cref{sect:model} is devoted to theoretical analysis of the $L_1/L_2$ model.   In \Cref{sect:algorithms}, we apply the ADMM to minimize the ratio of $L_1$ and $L_2$ with two variants of incorporating  a box constraint as well as applying on the image gradient. We conduct extensive experiments in \Cref{sect:experiments} to demonstrate the performance of the proposed approaches over the state-of-the-art in sparse recovery and MRI reconstruction. \Cref{sect:ratio} is a fun exercise, where we use the $L_1/L_2$ minimization to compute the right-hand-side of the NSP condition \eqref{eq:NSPratio}, leading to an empirical upper bound of the exact $L_1$ recovery guarantee.  
Finally, conclusions and future works are given in \Cref{sect:conclusion}.

\section{Rationales of the $L_1/L_2 $ model}\label{sect:model}
We begin with a toy example to illustrate the advantages of $L_1/L_2$ over other alternatives, followed by some theoretical properties of the proposed model.

\subsection{Toy example} 
Define a matrix $A $ as
\begin{equation}
A :=
\begin{bmatrix}
1 & -1 & 0 & 0 & 0 & 0
\\
1 & 0 & -1 & 0 & 0 & 0
\\
0 & 1 & 1 & 1 & 0 & 0
\\
2 & 2 & 0 & 0 & 1 & 0
\\
1 & 1 & 0 & 0 & 0 & -1
\end{bmatrix}
\in \mathbb{R}^{5 \times 6},
\end{equation}
and $\h b = (0,0,20,40,18)^{T} \in \mathbb{R}^{5} $. It is straightforward  that any general solutions of $A \h x =\h  b $ have the form of
$\h x = (t,t,t,20-2t,40-4t,2(t-9))^{T}$
for a scalar $t \in \mathbb{R} $. The sparsest solution occurs at $t=0$, where the sparsity of $\h x $ is 3 and some local solutions include  $t = 10 $ for sparsity being 4 and $t = 9 $ for sparsity being 5.  In  \Cref{fig:toy}, we plot various objective functions with respect to $t$, including $L_1, L_p$ (for $p=1/2$), $L_1$-$L_2$, and TL1 (for $a=1$ as suggested in \cite{zhangX18}). Note that all these functions are not differentiable at the values of $t = 0,9,$ and $10$, where the sparsity of $\h x$ is strictly smaller than 6. The sparsest vector $\h x$ corresponding to $t = 0$  can only be found by minimizing TL1  and $L_1/L_2 $, while the other models find $t = 10$ as a global minimum. 
%One drawback of TL1 is that the function is very narrow around the critical points, which implies it is sensitive to initial guess and difficult to find the global solution. In addition, TL1 has more local minimizers than our $L_1/L_2$ model (3 versus 2). 
%Note that we choose $a=1$ for the TL1 model as suggested in \cite{zhangX18}. Adjusting $a$ may lead to the global solution at $t=0,$ but this is beyond the scope of this paper. 

\begin{figure}
	\centering 
	\subfloat[$L_1$]{\label{fig:toy_l1}\includegraphics[width=0.33\textwidth]{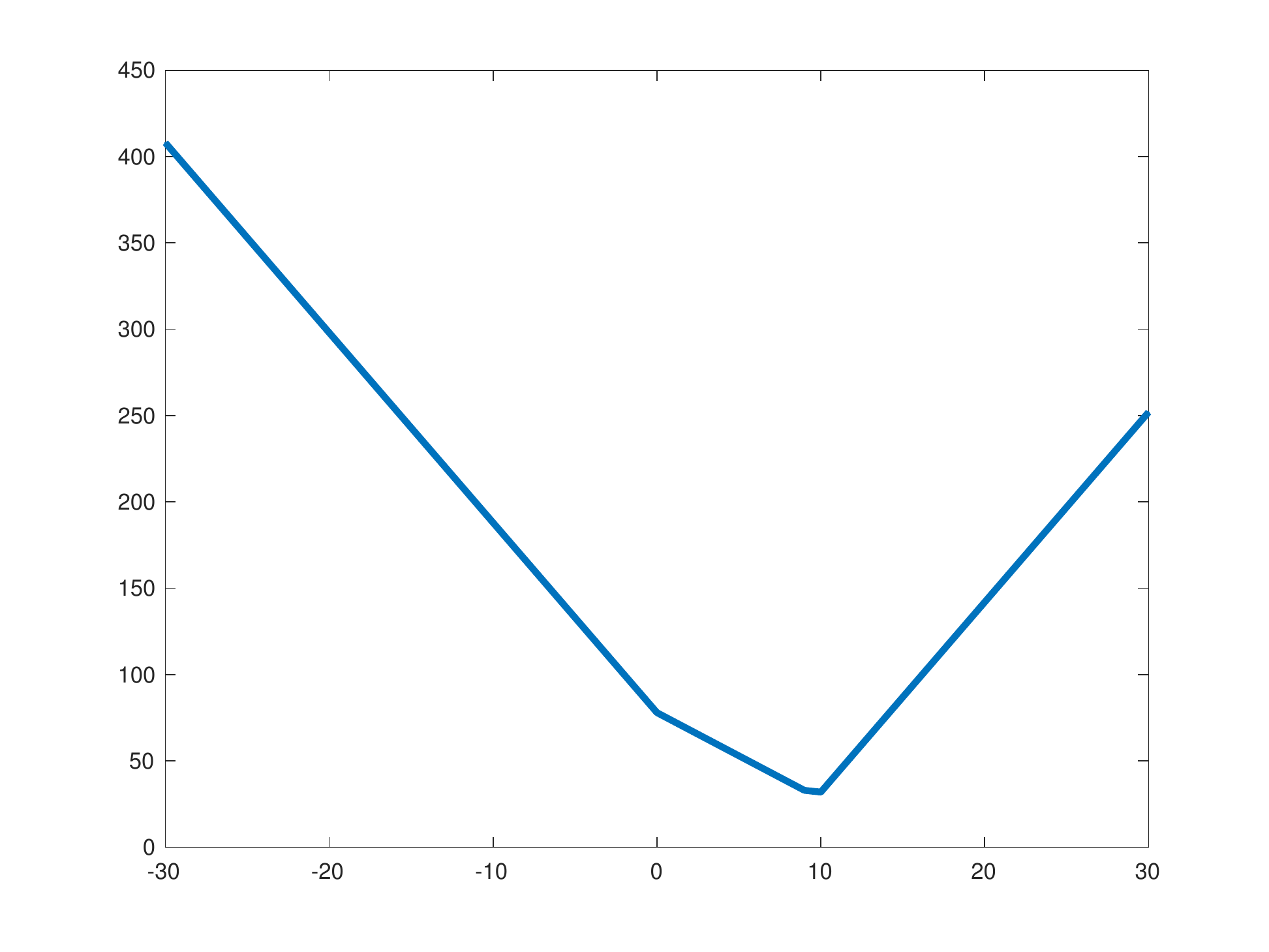}} 
		\subfloat[$L_p$ (p=1/2)]{\label{fig:toy_lp}\includegraphics[width=0.33\textwidth]{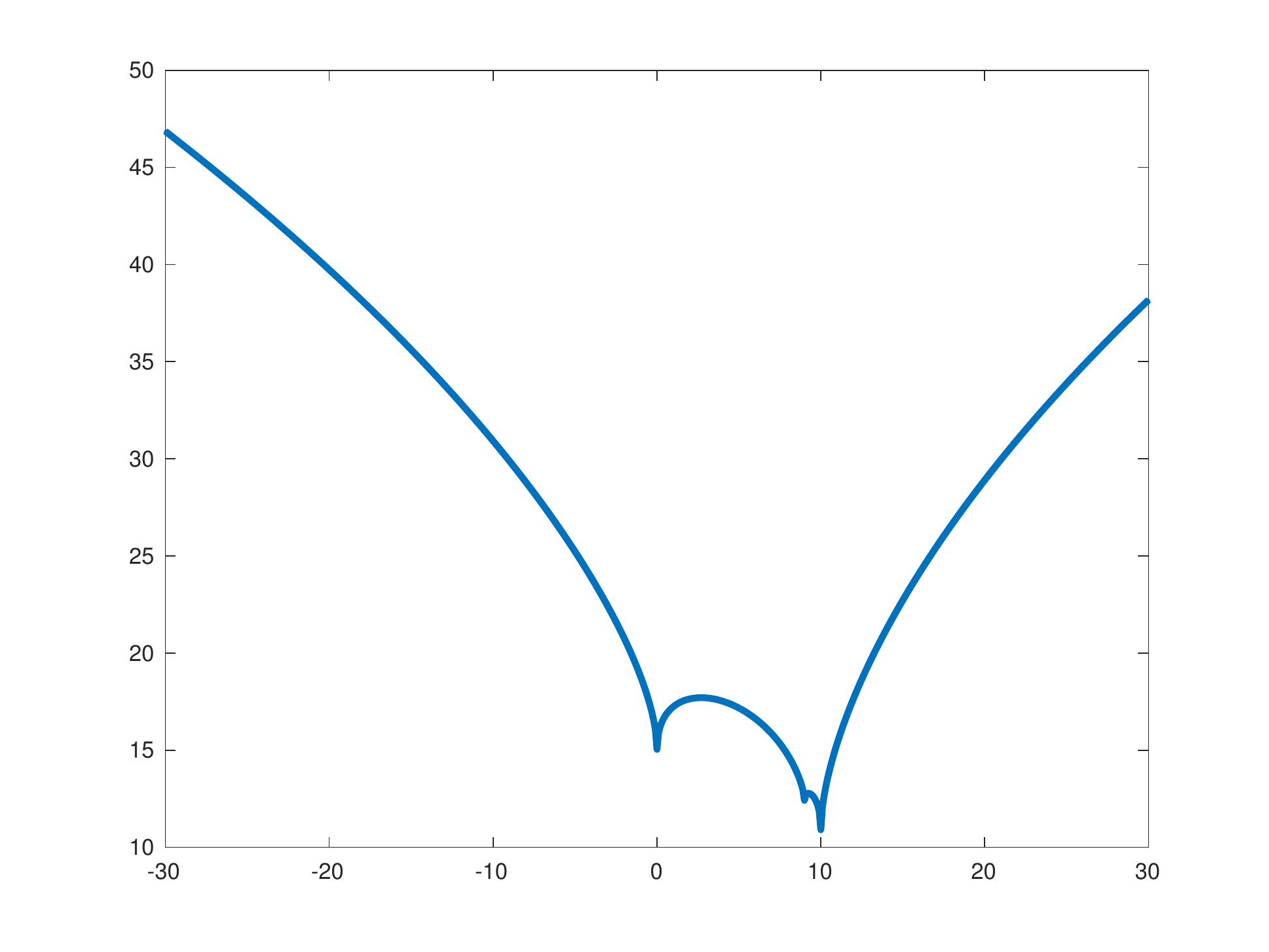}}\\
	\subfloat[TL1]{\label{fig:toy_tl1}\includegraphics[width=0.33\textwidth]{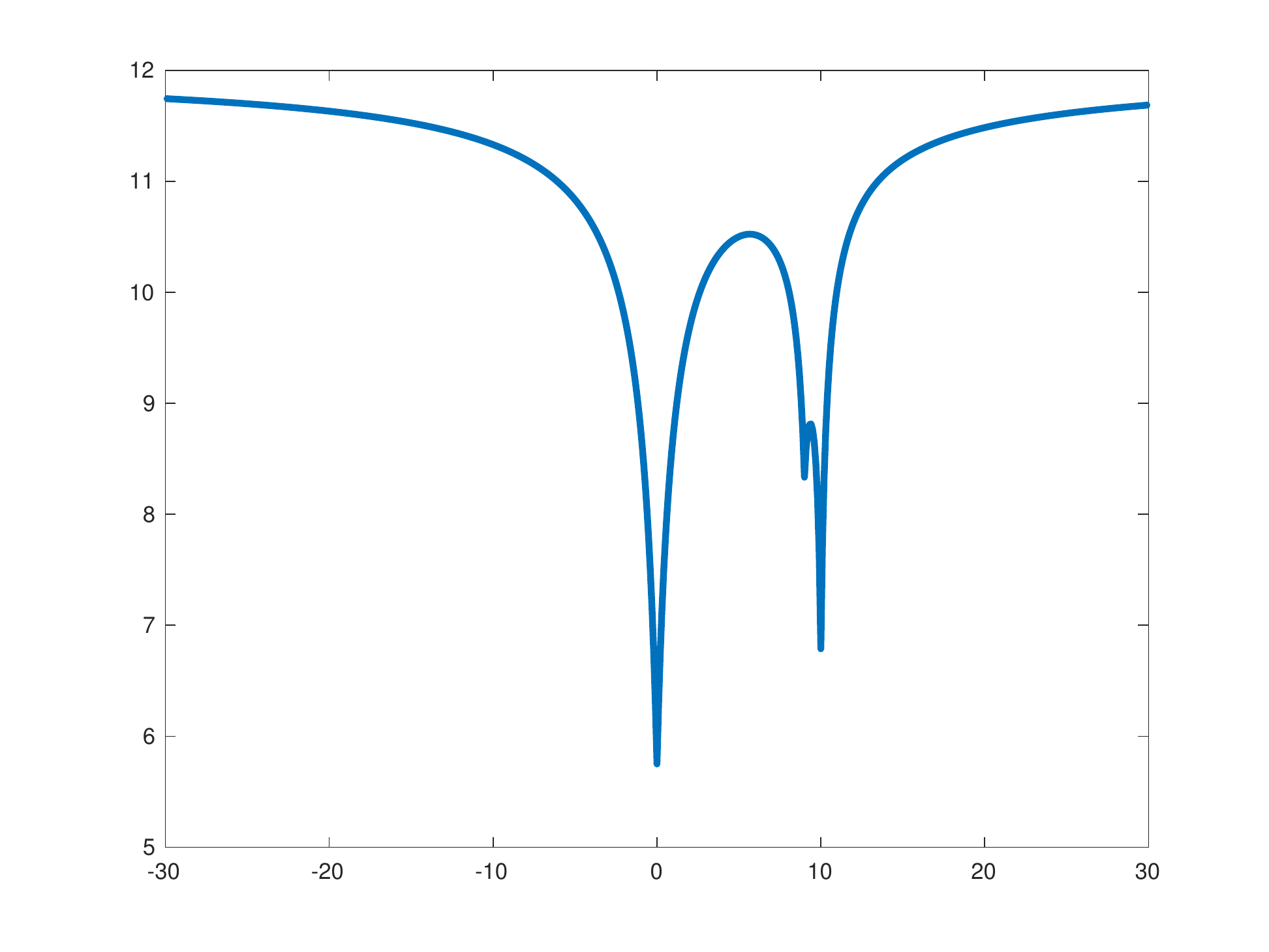}}
	\subfloat[$L_1$-$L_2$]{\label{fig:toy_l1ml2}\includegraphics[width=0.33\textwidth]{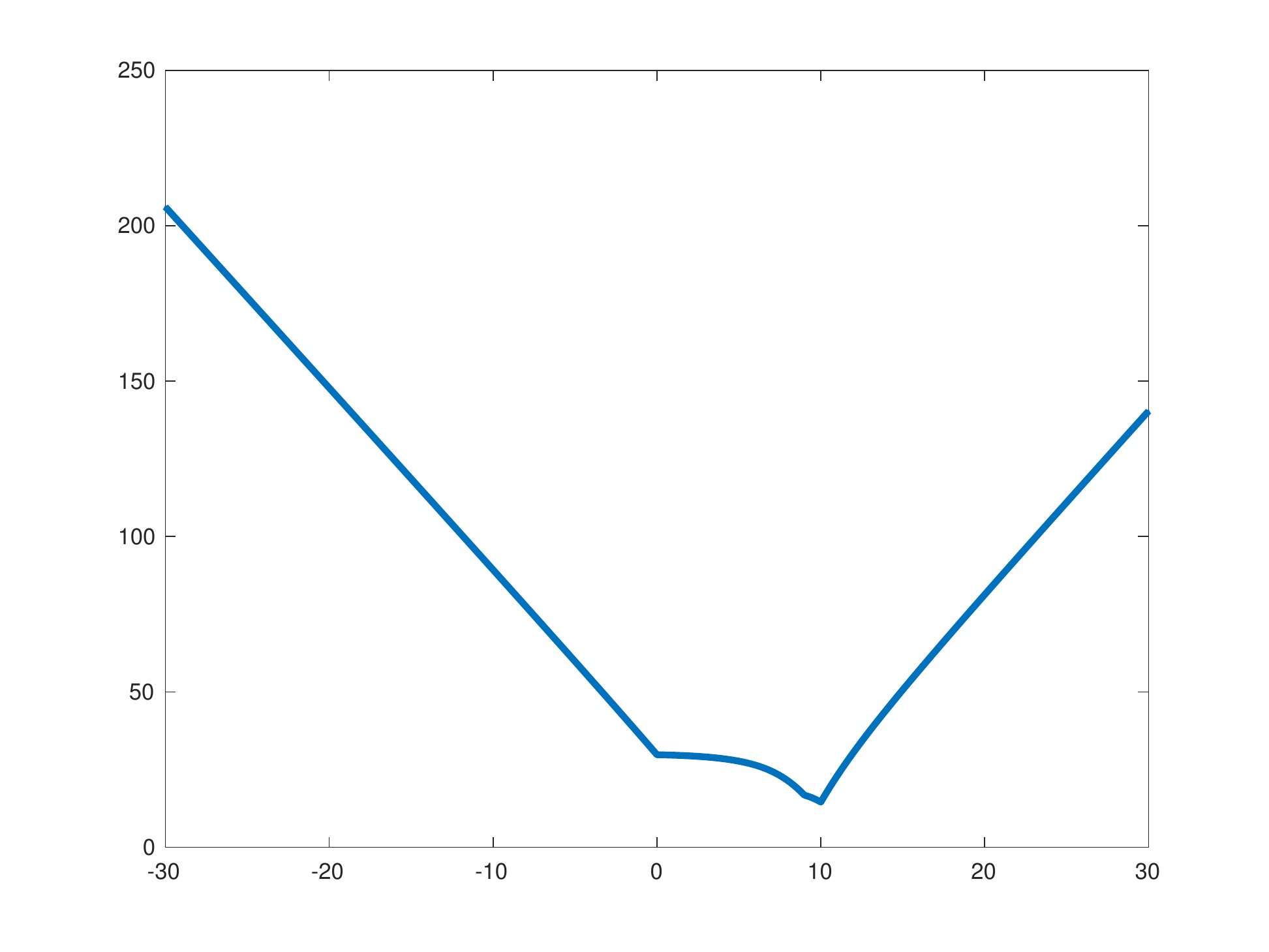}}
	\subfloat[$L_1/L_2$]{\label{fig:toy_l1dl2}\includegraphics[width=0.33\textwidth]{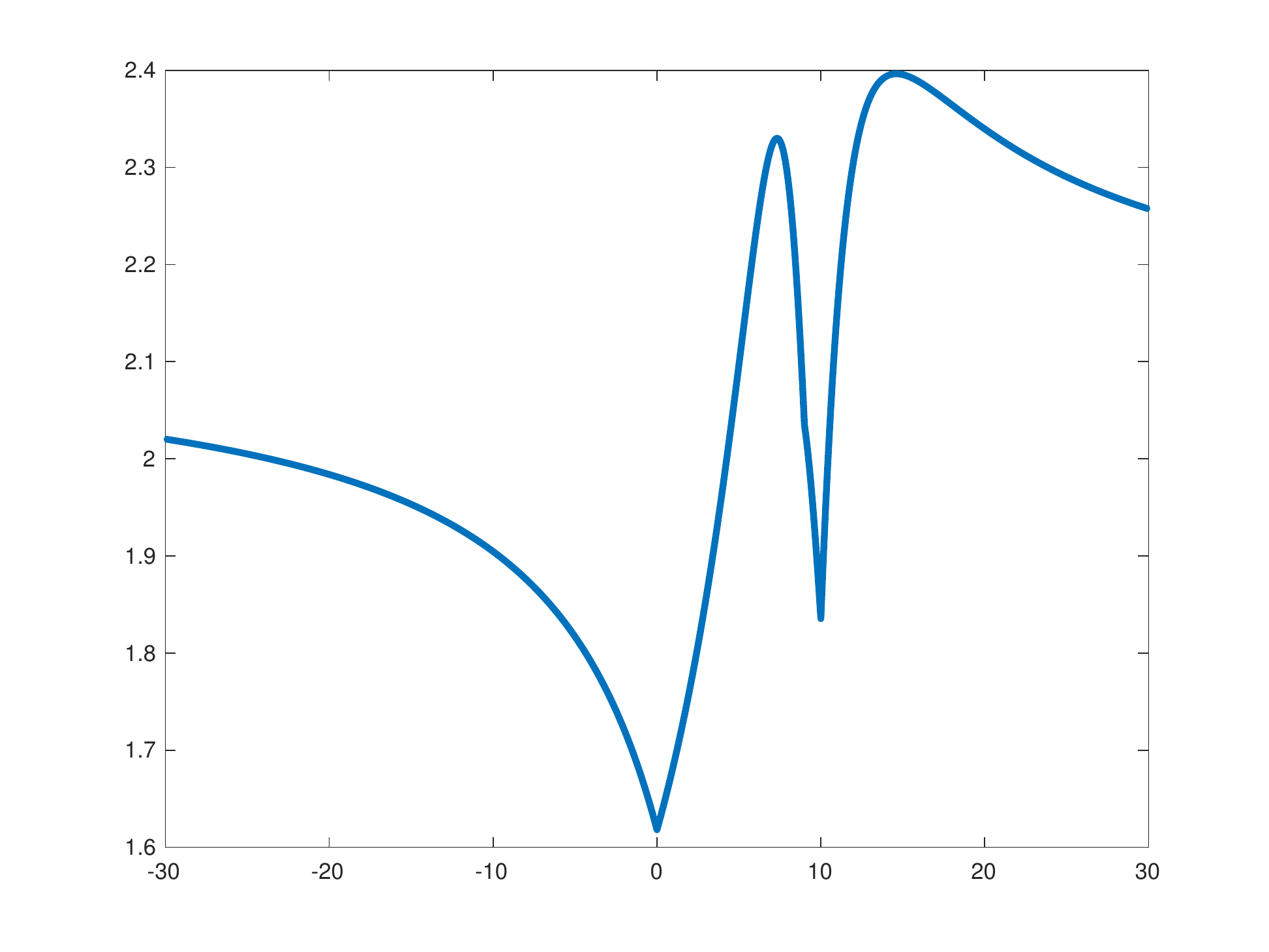}}
	\caption{The objective functions of a toy example illustrate that only $L_1/L_2$ and TL1 can find $t=0$ as the global minimizer, but TL1 has a very narrow basin of attraction (thus sensitive to initial guess and  difficult to find the global solution.). }
	\label{fig:toy}
\end{figure}

\subsection{Theoretical properties}

Recently, Tran and Webster \cite{tran2017unified} generalized the NSP  to deal with sparse promoting metrics that are symmetric, separable and concave, which unfortunately  does not apply to $L_1/L_2$ (not separable), but this work motivates us to  consider a stronger form of the NSP, as defined in \Cref{def:sNSP}.

\begin{definition}\label{def:sNSP}
	For any matrix $A \in \mathbb{R}^{m \times n} $, we say the matrix $A$ satisfies  a strong null space property  (sNSP)  of order $s$ if
%	\begin{equation}
%	\ker(A) \subset \left\lbrace \h  v  \in \mathbb{R}^n \, \middle| \, (s+1)\|\h v_S\|_1 \leq \|\h v_{\bar{S}}\|_1, \forall S \subset [n], |S| \leq s  \right\rbrace.
%	\end{equation}
	\begin{equation}
	(s+1)\left\|\h v_S\right\|_1 \leq \left\|\h v_{\bar{S}}\right\|_1, \ \h v \in \ker(A)  \backslash \{ \h 0 \}, \ \forall S \subset [n], \ |S| \leq s.
	\end{equation}
\end{definition}

Note that \Cref{def:sNSP} is stronger than the original NSP in \Cref{def:NSP} in the sense that  if a matrix satisfies sNSP then it  also satisfies  the original NSP.
The following theorem says that any $s$-sparse vector is a local minimizer of $L_1/L_2$ provided the matrix has the sNSP of order $s$. The proof is given in Appendix.

\begin{theorem}\label{thm:sNSP}
	Assume an $m\times n$ matrix $A$  satisfies the sNSP of order $s,$ then any $s$-sparse solution of $A\h x = \h b $ ($\h b \neq \h 0 $)  is a local minimum for $L_1/L_2$ in the feasible space of $A\h x =\h b $. i.e., there exists a positive number $t^* > 0 $ such that for every $\h v\in \ker(A) $ with $0<\|\h v\|_2\leq t^*$ we have
	\begin{equation}
	\frac{\|\h x\|_1}{\|\h x\|_2} \leq  \frac{\|\h x+\h v\|_1}{\|\h x+\h v\|_2}.
	\end{equation}
\end{theorem}

Finally, we show the optimal value of the $L_1/L_2$ subject to $A\h x=\h b$ is upper bounded by the same ratio with $\h b= \h 0$; see \Cref{prop:ratio}.  

\begin{prop}\label{prop:ratio}
	For any $A\in \mathbb{R}^{m\times n}, \h x \in \mathbb{R}^{n}, $ we have
	\begin{equation}\label{eq:prop}
	\inf_{\h z \in \Rbb^n} \left\lbrace \frac{\|\h z\|_1}{\|\h z\|_2} \,\middle| \, A \h z =A\h x \right\rbrace \leq \inf_{\h z \in \Rbb^n} \left\lbrace \frac{\|\h z\|_1}{\|\h z\|_2} \,\middle| \, \h z \in \ker(A) \setminus \{\h 0\} \right\rbrace. 
	\end{equation}
\end{prop}

\begin{proof}
	Denote 
	\begin{equation}\label{eq:l1dl2_alpha}
	\alpha^* = \inf_{\h z \in \Rbb^n} \left\lbrace \frac{\|\h z\|_1}{\|\h z\|_2} \,\middle| \, A \h z =A\h x \right\rbrace.
	\end{equation}
	For every $\h v \in \ker(A) \setminus \{\h 0\} $ and  $ t \in \Rbb $, we have that
	\begin{equation}
	\alpha^*  \leq  \frac{\|\h x + t \h v\|_1}{\|\h x + t \h v\|_2},
	\end{equation}
	since $A(\h x + t \h v) = \h b$. Then we obtain
	\begin{equation}
	\lim_{t \rightarrow \infty} \frac{\|\h x + t \h v\|_1}{\|\h x + t \h v\|_2} = \lim_{t \rightarrow \infty} \frac{\|\h x/t + \h v\|_1}{\|\h x/t + \h v\|_2} = \frac{\| \h v\|_1}{\| \h v\|_2}. 
	\end{equation}
	Therefore, for every $\h v \in \ker(A) \setminus \{\h 0\}, $ 
	\begin{equation}
	\alpha^*  \leq  \frac{\|\h v\|_1}{\|\h v\|_2}, 
	\end{equation}
	which directly leads to the desired inequality \eqref{eq:prop}.  
\end{proof}

\Cref{prop:ratio} implies that the left-hand-side of the inequality involves both the underlying signal $\h x$ and the system matrix $A$, which can be upper bounded by the minimum ratio that only involves  $A$. %This relationship will be numerically checked in \Cref{sect:ratio}.

\section{Numerical schemes} \label{sect:algorithms}
The proposed model is 
\begin{equation}\label{equ:l1dl2_model}
\min\limits_{\mathbf{x}\in \mathbb{R}^n} \ \ \left\{ \frac{\|\h x\|_1}{ \|\h x\|_2} +  I_0(A\h x-\h b) \right\},
\end{equation}
where $I_{S}(\h t)$ is the function enforcing $\h t$ into the feasible set $S$, i.e., 
	\begin{equation}\label{eq:indicator}
	I_S(\h t) = 
	\begin{cases}
	0	&	 \h t \in S,
	\\
	+\infty	&	\text{otherwise}.
	\end{cases}
	\end{equation}
In  \Cref{sect:l1dl2_admm}, we detail the ADMM algorithm for minimizing \eqref{equ:l1dl2_model}, followed by a minor change to incorporate  additional box constraint in \Cref{sect:box}. 
We discuss in \Cref{sect:l1dl2_gradient} another variant of $L_1/L_2$ on the gradient to deal with imaging applications. 

\subsection{The $L_1/L_2$ minimization via ADMM}\label{sect:l1dl2_admm}

In order to apply the ADMM \cite{boydPCPE11admm} to solve for \eqref{equ:l1dl2_model}, we  introduce two auxiliary variables and  rewrite \eqref{equ:l1dl2_model} into an equivalent form,
\begin{equation}\label{eq:l1dl2_admm}
\min_{\h x,\h y,\h z} \ \ \left\{ \frac{\|\h z\|_1}{ \|\h y\|_2} +  I_0(A\h x-\h b) \right\} \quad \mathrm{s.t.} \ \ \h x =\h y, \ \ \h x =\h z.
\end{equation}
The augmented Lagrangian for \eqref{eq:l1dl2_admm} is 
\begin{eqnarray}
L_{\rho_1,\rho_2}(\h x,\h y,\h z;\h v,\h w) &=& \frac{\|\h z\|_1}{ \|\h y\|_2} + I_0(A\h x-\h b)+   \left\langle \h v, \h x-\h y \right\rangle + \frac{\rho_1}{2}\|\h x-\h y\|_2^2 +  \left\langle \h w, \h x-\h z \right\rangle + \frac{\rho_2}{2}\|\h x-\h z\|_2^2
\notag\\ &=
&\frac{\|\h z\|_1}{ \|\h y\|_2} + I_0(A\h x-\h b)+  \frac{\rho_1}{2}\left\|\h x-\h y + \frac{1}{\rho_1} \h v\right\|_2^2 +   \frac{\rho_2}{2}\left\|\h x-\h z + \frac{1}{\rho_2} \h w \right\|_2^2.
\end{eqnarray}

The ADMM consists of the following five steps:
\begin{equation}\label{ADMM_contrainted}
\left\{\begin{array}{l}
 \h x^{(k+1)} := \arg\min\limits_{\h x}L_{\rho_1,\rho_2}(\h x,\h y^{(k)},\h z^{(k)};\h v^{(k)},\h w^{(k)}),\\
\h y^{(k+1)} := \arg\min\limits_{\h y}L_{\rho_1,\rho_2}(\h x^{(k+1)},\h y,\h z^{(k)};\h v^{(k)},\h w^{(k)}),\\
\h z^{(k+1)} := \arg\min\limits_{\h z}L_{\rho_1,\rho_2}(\h x^{(k+1)},\h y^{(k+1)},\h z;\h v^{(k)},\h w^{(k)}),\\
\h v^{(k+1)} := \h v^{(k)}+\rho_1(\h x^{(k+1)}-\h y^{(k+1)}),\\
\h w^{(k+1)} := \h w^{(k)}+\rho_2(\h x^{(k+1)}-\h z^{(k+1)}).
\end{array}\right.
\end{equation}

%\begin{equation}\label{eq:L1ADMM}
%\left\{\begin{array}{l}
%\h x^{(k+1)} = \arg\min\limits_{\h x}L_\rho(\h x,\h y^{(k)};\h u^{(k)}) = \arg \min\limits_{\h x} \|\h x\|_1+\frac \rho 2 \|\h x-\h y^{(k)}+\frac {\h u^{(k)}} \rho\|_2^2, \\
%\h y^{(k+1)} = \arg\min\limits_{\h y}L_\rho(\h x^{(k+1)},\h y;\h u^{(k)})=\arg\min\limits_{\h y} I(A\h y-\h b)+\frac \rho 2 \|\h x^{(k+1)}-\h y+\frac {\h u^{(k)}} \rho\|_2^2, \\
%\h u^{(k+1)} = \h u^{(k)}+\rho(\h x^{(k+1)}-\h y^{(k+1)}).
%\end{array}\right.
%\end{equation}

The update for $\h x$ is  a projection to the affine space of $A\h x=\h b$, 
\begin{equation*}
\begin{split}
\h x^{(k+1)} := &  \arg\min_{\h x}L_{\rho_1,\rho_2}(\h x,\h y^{(k)},\h z^{(k)};\h v^{(k)},\h w^{(k)})
\\ =&
%\arg\min_{\h x} \left\{ I(A\h x-\h b)+  \frac{\rho_1}{2}\left\|\h x-\h y^{(k)} + \frac{1}{\rho_1} \h v^{(k)}\right\|_2^2 +   \frac{\rho_2}{2}\left\|\h x-\h z^{(k)} + \frac{1}{\rho_2} \h w^{(k)}\right\|_2^2 \right\}
% \\ =&
\arg\min_{\h x} \left\{\frac{\rho_1 + \rho_2}{2}\left\|\h x- \h f^{(k)} \right\|_2^2 \ \ \mathrm{s.t.} \ \ A\h x=\h b \right\}
\\
= & \left( I - A^{T}(AA^{T})^{-1}A \right)  \h f^{(k)}  + A^{T}(AA^{T})^{-1}\h b,
\end{split}
\end{equation*}
where
\begin{equation}
\h f^{(k)} = \frac{\rho_1}{\rho_1 + \rho_2}\left(\h y^{(k)} - \frac{1}{\rho_1} \h v^{(k)}\right) + \frac{\rho_2}{\rho_1 + \rho_2}
\left(\h z^{(k)} - \frac{1}{\rho_2} \h w^{(k)}\right).
\end{equation}

As for the $\h y$-subproblem,  let $c^{(k)} = \|\h z^{(k)}\|_1$ and $\h d^{(k)} = \h x^{(k+1)} + \frac {\h v^{(k)}}{ \rho_1}$ and the minimization subproblem
%\eqref{eq:ADMM-y}
 reduces to
\begin{equation}
\h y^{(k+1)} = \arg\min_{\h y} \frac{c^{(k)}} { \|\h y\|_2} + \frac{\rho_1}{2} \|\h y-\h d^{(k)}\|_2^2.
\end{equation}
If $\h d^{(k)} = 0 $ then any vector $\h  y $ with $\|\h y\|_2 = \sqrt[3]{\frac{c^{(k)}}{\rho_1}} $ is a solution to the minimization problem. If $c^{(k)} = 0 $ then $\h y =\h  d^{(k)} $ is the solution. Now we consider $\h d^{(k)} \neq 0 $ and $ c^{(k)} \neq 0 $. By taking derivative of the objective function with respect to $\h y$, we obtain
\[
\left(-\frac {c^{(k)} }{\|\h y\|_2^3} + \rho_1\right)\h y = \rho_1 \h d^{(k)}. 
\]
As a result,  there exists a positive number $\tau^{(k)} \geq 0$ such that $
\h y = \tau^{(k)} \h d^{(k)}.$ Given $\h d^{(k)}$, we denote $\eta^{(k)} = \|\h d^{(k)}\|_2$. For $\eta^{(k)}>0$, finding $\h y$ becomes a one-dimensional search for the parameter $\tau^{(k)}$. In other words, if we take $D^{(k)} = \frac{c^{(k)}}{\rho_1 (\eta^{(k)})^3},$ then $\tau^{(k)}$ is a root of
\[
\underbrace{ \tau^3 - \tau^2 - D^{(k)}}_{F(\tau)} = 0.
\]
The cubic-root formula suggests that   $F(\tau) = 0$ has only one real root, which can be found by the following closed-form solution.
\begin{equation}\label{eq:update_root}
\tau^{(k)} = \frac{1}{3} + \frac{1}{3}(C^{(k)} + \frac{1}{C^{(k)}}),  \ \mbox{with} \ C^{(k)} = \sqrt[3]{ \frac{27D^{(k)} + 2 + \sqrt{(27D^{(k)}+2)^2 - 4}}{2} }. 
\end{equation}
In summary, we have
\begin{equation}
\h y^{(k+1)} =
\begin{cases}
\h e^{(k)}	& \h d^{(k)} = 0,
\\
\tau^{(k)} \h d^{(k)} & \h d^{(k)} \neq 0,
\end{cases}
\end{equation}
where $\h e^{(k)} $ is a random vector with the $L_2$ norm to be $\sqrt[3]{\frac{c^{(k)}}{\rho_1}} $.

Finally, the ADMM update for $\h z $ is 
%given by soft shrinkage operator
\begin{equation}\label{eq:l1dl2admm_z}
\h z^{(k+1)} = \mathbf{shrink}\left(\h x^{(k+1)} + \frac{\h w^{(k)}}{\rho_2}, \frac{1}{\rho_2 \|\h y^{(k+1)} \|_2}\right), 
\end{equation}
where $\mathbf{shrink}$ is often referred to as \textit{soft shrinkage} operator,
\begin{equation}
\label{eq:shrinkage}
	\mathbf{shrink}(\h v, \mu)_i = \mathrm{sign}(v_i)\max\left(|v_i|-\mu, 0\right), \quad i = 1, 2, \dots, n. 
\end{equation}
We summarize the ADMM algorithm for solving the $L_1/L_2$ minimization problem in \Cref{alg:l1dl2ADMM}. 
\begin{algorithm}
\caption{The $L_1/L_2$ minimization via ADMM. }\label{alg:l1dl2ADMM}
\begin{algorithmic}[0]
\STATE{Input: $A\in \mathbb{R}^{m\times n}, \h b\in \mathbb{R}^{m\times 1}$, Max and $ \epsilon \in \mathbb{R}$}
\WHILE{$k < $ Max or $\|\h x^{(k)}-\h x^{(k-1)}\|_2/\|\h x^{(k)}\| > \epsilon$}
\STATE{$ \h x^{(k+1)} = \left( I - A^{T}(AA^{T})^{-1}A \right)  \h f^{(k)}  + A^{T}(AA^{T})^{-1}\h b$}
\STATE{$\h y^{(k+1)} =
 \begin{cases}
\h e^{(k)}	& \h d^{(k)} = 0
\\
\tau^{(k)} \h d^{(k)} & \h d^{(k)} \neq 0
\end{cases}$ }
\STATE{$\h z^{(k+1)} = \mathbf{ shrink}\left(\h x^{(k+1)} + \frac{\h w^{(k)}}{\rho_2}, \frac{1}{\rho_2 \|\h y^{(k+1)} \|_2}\right) $}
\STATE{$\h v^{(k+1)} = \h v^{(k)}+\rho_1(\h x^{(k+1)}-\h y^{(k+1)})$}
\STATE{$\h w^{(k+1)} = \h w^{(k)}+\rho_2(\h x^{(k+1)}-\h z^{(k+1)})$}
\STATE{k = k+1}
\ENDWHILE
\RETURN $\h x^{(k)}$ \end{algorithmic} 
\end{algorithm}

\begin{remark}
		We can pre-compute the matrix $I - A^{T}(AA^{T})^{-1}A $ and the vector $A^{T}(AA^{T})^{-1}\h b$  in \Cref{alg:l1dl2ADMM}. The complexity is $O(m^2n)$ for the pre-computation including the matrix-matrix multiplication and Cholesky decomposition for solving linear system. In each iteration, we need to do matrix-vector multiplication for the $\h x$-subproblem, which is in the order of $O(n^2)$. In the $\h y$-subproblem, the rooting finding is one-dimensional search, whose cost can be neglected. The $\h z$-subproblem is pixel-wise shrinkage operation and only takes $O(n)$.  In summary, the  computation complexity for each iteration is $O(n^2)$.   We can consider the parallel computing to further speed up, thanks to the separation of the $\h z$-subproblem.
	\end{remark}

\subsection{$L_1/L_2$ with box constraint}\label{sect:box}
The $L_1/L_2$ model has an intrinsic drawback that tends to produce one erroneously large coefficient while suppressing the other non-zero elements, under which case the ratio is reduced. To compensate for this drawback, it is helpful to incorporate a box constraint, if we know lower/upper bounds of the underlying signal \textit{a priori}. Specifically, we propose
\begin{equation}\label{equ:l1dl2_box}
\min\limits_{\h x\in \mathbb{R}^n} \ \ \left\{ \frac{\|\h x\|_1}{ \|\h x\|_2} +  I_0(A\h x-\h b)\middle| \ \h x\in[c,d] \right\},
\end{equation}
which is referred to as  $L_1/L_2$-box. Similar to \eqref{eq:l1dl2_admm}, we look at the following form that enforces  the box constraint on variable $\h z$, 
\begin{equation}\label{eq:l1dl2_admm_box}
\min_{\h x,\h y,\h z} \ \ \left\{ \frac{\|\h z\|_1}{ \|\h y\|_2} +  I_0(A\h x-\h b) \right\} \quad \mathrm{s.t.} \ \ \h x =\h y, \ \ \h x =\h z, \ \ \h z\in [c,d].
\end{equation}
The only change we need to make by adapting \Cref{alg:l1dl2ADMM} to the $L_1/L_2$-box is 
the $\h z$ update. The $\h z$-subproblem in \eqref{ADMM_contrainted} with the box constraint is 
\begin{equation}
\label{equ:sub_p_z}
\min_{\h z}\frac 1{\|\h y^{(k+1)}\|_2} \|\h z\|_1 + \frac {\rho_2} 2 \|\h x^{(k+1)}-\h z+\frac 1 {\rho_2} \h w^{(k)}\|_2^2 \quad  \mbox{s.t.} \quad  \h z \in [c,d].
\end{equation}
For a convex problem \eqref{equ:sub_p_z} involving the $L_1$ norm, it has a closed-form solution given by the soft shrinkage, followed by projection to the interval $[c, d]$. 
In particular, simple calculations show that
%\begin{equation}
%\h z^{(k+1)} = \left\{
%\begin{array}{ll}
%\h v-\mu_1 & \h v>\mu_1, \\
%0 & \mu_2\leq\h v\leq \mu_1, \\
%\h v+\mu & \h v<\mu_2. 
%\end{array}
%\right.
%\end{equation}
\begin{equation}
		z^{(k+1)}_i = \min\left\{\max(\hat{z}_i, c), d\right\}, \quad i = 1, 2, \dots, n,  
	\end{equation}
%where $\h v = \h x^{(k+1)} + \frac{\h w^{(k)}}{\rho_2}, \mu_1 = \min(\frac{1}{\rho_2 \|\h y^{(k+1)} \|_2},d), $ and $\mu_2=\max(-\frac{1}{\rho_2 \|\h y^{(k+1)} \|_2},c)$. 
where $\hat{\h z} =  \mathbf{shrink}\left( \h r, \nu \right)$, $\h r = \h x^{(k+1)} + \frac{\h w^{(k)}}{\rho_2}$ and $\nu = \frac{1}{\rho_2 \|\h y^{(k+1)}\|}$. 
%{\color{red}[please check.]}
If the box constraint $[c,d]$ is symmetric, i.e., $c = -d$ and $d>0$, it follows from \cite{beck2017first} that  the update for $\h z$ can be expressed as
%\begin{equation}
%\h z^{(k+1)} = \mathbf{shrink}\left(\h v, \min\left(\frac{1}{\rho_2 \|\h y^{(k+1)} \|_2},d\right)\right).
%\end{equation}
\begin{equation}
		z^{(k+1)}_i = \mathrm{sign}(v_i) \min \left\{\max(|r_i|-\nu, 0), d \right\}, \quad i = 1, 2, \dots, n. 
	\end{equation}

\begin{remark}
%	We carefully check the literature regarding ADMM convergence, e.g., . Unfortunately, none of
	 The existing literature on the ADMM convergence  \cite{guo2017convergence,hong2016convergence,li2015global,pang2018decomposition,wang2018convergences,wang2014convergence,wang2019global} requires the existence of one separable function in the objective function, whose gradient is Lipschitz continuous.  Obviously, $L_1/L_2$ does not satisfy this assumption, no matter with or without the box constraint. Therefore, we have difficulties in analyzing the convergence theoretically. Instead, we show the convergence empirically in \Cref{sect:experiments} by plotting  residual errors and objective functions, which gives strong supports for theoretical analysis in the future.
	\end{remark}

\subsection{$L_1/L_2$ on the gradient}\label{sect:l1dl2_gradient}

We adapt the $L_1/L_2$ model to apply on the gradient, which enables us to deal with imaging applications. Let $u\in\mathbb R^{n\times m}$ be an underlying image of size $n\times m$. Denote $A$ as a linear operator that models a certain degradation process to obtain the measured data $f$. For example, $A$ can be a subsampling operator in the frequency domain and recovering $u$ from $f$  is called MRI reconstruction. In short, the proposed gradient model is given by
\begin{equation}\label{eq:grad_con}
\min_{u\in\mathbb R^{n\times m}}  \frac{\| \nabla u \|_1}{\| \nabla u \|_2} \quad \mathrm{s.t.} \ \  A u =f, \ u\in [0, 1],
\end{equation}
where $\nabla$ denotes discrete gradient operator $ \nabla u:=
\{[u_{ij}-u_{(i+1)j}]_{i=1}^{n}\}_{j=1}^m, \{[u_{ij}-u_{i(j+1)}]_{j=1}^{m}\}_{i=1}^n
$ with periodic boundary condition; hence the model is referred to as $L_1/L_2$-grad.
 Note that the box constraint $0\leq u\leq 1$ is a reasonable assumption in the MRI reconstruction problem.

To solve for \eqref{eq:grad_con}, we introduce three auxiliary variables $\h d, \h h,$ and $v$, leading to an equivalent problem,
\begin{equation}\label{eq:grad_con2}
\min_{u\in\mathbb R^{n\times m}}  \frac{\| \h d \|_1}{\| \h h \|_2} \quad \mathrm{s.t.} \ \  A u =f, \  \h d = \nabla u,\  \h h = \nabla u, u = v, 0\leq v\leq 1.
\end{equation}
Note that we denote $\h d$ and  $\h h$ in bold to indicate that they have two components corresponding to both $x$ and $y$ derivatives. 
The augmented Lagrangian  is expressed as 
\begin{equation}\label{eq:grad_aug}
\begin{split}
\mathcal{L}(u,\h d, \h h, v; w, \h b_1, \h b_2,e ) &= \frac{\|\h d\|_1 }{\| \h h\|_2} + \frac{\lambda}{2}\|Au - f-w\|_2^2 + \frac{\rho_1}{2} \|\h d-\nabla u-\h b_1 \|_2^2 \\ & + \frac{\rho_2}{2} \|\h h-\nabla u -\h b_2\|_2^2 + \frac{\rho_3}{2} \|v-  u - e\|_2^2 + I_{[0, 1]}(v),
\end{split}
\end{equation}
where $w, \h b_1, \h b_2, e$ are dual variables and $\lambda, \rho_1, \rho_2,\rho_3$ are positive parameters. 
The updates for $\h d,\h h$ are the same as Algorithm~\ref{alg:l1dl2ADMM}. Specifically for $\h h $, we consider $D^{(k)} = \frac{\|\h d \|_1}{\rho_2 \|\nabla u^{(k+1)} + \h g^{(k)}\|_2^3}$ and hence $\tau^{(k)} $ is the root of the same polynomial as in \eqref{eq:update_root}. By taking derivative of \eqref{eq:grad_aug} with respect to $u$, we can obtain the $u$-update, i.e.,
\begin{equation}\label{eq:grad_u_update}
\begin{split}
	u^{(k+1)} =& \left(\lambda A^TA-(\rho_1+\rho_2)\triangle + \rho_3 I\right)^{-1}\left(\lambda A^T (f+w^{(k)}) \right.\\ &\left.+ \rho_1 \nabla^T(\h d^{(k)} -\h b_1^{(k)})+\rho_2\nabla^T(\h h^{(k)}-\h b_2^{(k)})+\rho_3(v^{(k)}-e^{(k)})\right).
\end{split}
\end{equation}
Note for certain operator $A$, the inverse in the $u$-update \eqref{eq:grad_u_update} can be computed efficiently via the fast Fourier transform (FFT). The $v$-subproblem is a projection to an interval $[0, 1]$, i.e.,
\begin{equation}\label{equ:v_update}
v_{ij}^{(k+1)}=\min\left\{\max(u_{ij}^{(k+1)}+e_{ij}^{(k)}, 0), 1\right\}, \quad i = 1, 2, \dots, n, j = 1, 2, \dots, m.
\end{equation}
In summary, we present the ADMM algorithm for the  $L_1/L_2$-grad model in \Cref{alg:l1dl2gradADMM}.  
\begin{algorithm}
	\caption{The $L_1/L_2$-grad minimization via ADMM. }\label{alg:l1dl2gradADMM}
	\begin{algorithmic}[0]
		\STATE{Input: $f \in \mathbb{R}^{n\times m}$, $ A$, Max and $ \epsilon \in \mathbb{R}$.}
		\WHILE{$k < $ Max or $\| u^{(k)}-u^{(k-1)}\|_2/\|u^{(k)}\| > \epsilon$}
		\STATE{Solve $ u^{(k+1)}$ via \eqref{eq:grad_u_update}}
		\STATE{Solve $ v^{(k+1)}$ via \eqref{equ:v_update}}
		\STATE{$\h h^{(k+1)} =
			\begin{cases}
			\h e^{(k)}	& \nabla u^{(k+1)} + \h g^{(k)} = 0,
			\\
			\tau^{(k)} \left( \nabla u^{(k+1)} + \h g^{(k)} \right) & \nabla u^{(k+1)} + \h g^{(k)} \neq 0.
			\end{cases}$ }
		\STATE{$\h d^{(k+1)} = \mathbf{ shrink}\left(\nabla u^{(k+1)} + \h b^{(k)}, \frac{1}{\rho_1 \|\h h^{(k+1)} \|_2}\right) $}
		\STATE{$\h b^{(k+1)} = \h b^{(k)} + \nabla u^{(k+1)} - \h d^{(k+1)}$}
		\STATE{$\h g^{(k+1)} = \h g^{(k)} + \nabla u^{(k+1)} - \h h^{(k+1)}$}
		\STATE{$w^{(k+1)} = w^{(k)} + f - A u^{(k+1)}$}
		\STATE{$e^{(k+1)} = e^{(k)} + u^{(k+1)} - v^{(k+1)}$}
		\STATE{$k = k+1$}
		\ENDWHILE
		\RETURN $u^{(k)}$ \end{algorithmic} 
\end{algorithm}

\section{Numerical experiments}\label{sect:experiments}

In this section, we carry out a series of numerical tests
to demonstrate the performance of the proposed $L_1/L_2$ models together with its corresponding algorithms. 
All the numerical experiments are conducted on a standard desktop with CPU (Intel i7-6700, 3.4GHz) and $\mathrm{MATLAB \ 9.2 \ (R2017a)}. $

We consider two types of sensing matrices: one is called oversampled discrete cosine transform (DCT) and the other is Gaussian matrix. Specifically for the oversampled DCT, we follow the works of
\cite{DCT2012coherence,louYHX14,yinLHX14} to define  $A= [\h a_1, \h a_2, \dots, \h a_n]\in \mathbb{R}^{m\times n}$ with 
\begin{equation}\label{eq:oversampledDCT}
\h a_j := \frac{1}{\sqrt{m}}\cos \left(\frac{2\pi \h w j}{F} \right), \quad j = 1, \dots, n,
\end{equation}
where $\h w$ is a random vector uniformly distributed in $[0, 1]^m$ and $F\in \Rbb_+$  controls the coherence in a way that a larger  value of $F$ yields a more coherent matrix. 
In addition, we use $\mathcal N(\h 0,\Sigma)$ (the multi-variable normal distribution)  to generate Gaussian matrix, where the covariance matrix is $\Sigma = \{(1-r)*I(i=j)+r\}_{i,j}$ with a positive parameter $r$. This type of matrices is used in the TL1 paper \cite{zhangX18}, which mentioned that a larger $r$ value indicates a more difficult problem in sparse recovery. 
Throughout the experiments, we consider sensing matrices of size $64\times 1024$. 
The ground truth $\h x\in \mathbb{R}^{n}$ is simulated as $s$-sparse signal, where  $s$ is the total number of nonzero entries. The support of $\h x$ is a random index set and the values of non-zero elements follow Gaussian normal distribution i.e., $ (\h x_s)_i \sim \mathcal{N}(0,1), \quad i = 1,2,\dots, s.$ We then normalize the ground-truth signal to have maximum magnitude as 1 so that we can examine the performance of additional $[-1,1]$ box constraint.  

Due to the non-convex nature of the proposed $L_1/L_2$ model, the initial guess $\h x^{(0)}$ is very important and should be well-chosen. A typical choice is the $L_1$ solution \eqref{eq:l1}, which is used here. 
%Although the $L_1$ minimization can be solved by the ADMM (as described in \Cref{sect:admm}), 
We adopt a commercial optimization software called  Gurobi \cite{optimization2014inc} to minimize the $L_1$ norm via linear programming for the sake of efficiency.  The stopping criterion is when the relative error of $\h x^{(k)}$ to $\h x^{(k-1)}$ is smaller than  $10^{-8}$ or iterative number exceeds $10n$.

\subsection{Algorithmic behaviors}

We empirically  demonstrate the convergence of the proposed ADMM algorithms in \Cref{fig:distance_auxiliary}.  
Specifically we examine the $L_1/L_2$ minimization problem \eqref{equ:l1dl2_model}, where the sensing matrix is  an oversampled DCT matrix with $F=10$ and ground-truth sparse vector has 12 non-zero elements. We also study  the MRI reconstruction from 7 radical lines as a particular sparse gradient problem
that involves the $L_1/L_2$-grad minimization of \eqref{eq:grad_con} by \Cref{alg:l1dl2gradADMM}.

There are two auxiliary variables $\h y$ and $\h z$ in  $L_1/L_2$ such that $\h x = \h y=\h z$,
while two auxiliary variables $\h d, \h h$ are in $L_1/L_2$-grad for $\nabla u=\h d = \h h$. We show in the top row of \Cref{fig:distance_auxiliary}   the values of $\left\|\h x^{(k)}- \h y^{(k)}\right\|_2$ and $\left\|\h x^{(k)}- \h z^{(k)}\right\|_2$ as well as $\left\|\nabla u^{(k)}-\h d^{(k)}\right\|_2$ and $\left\|\nabla u^{(k)}-\h h^{(k)}\right\|_2$, all are plotted with respect to the iteration counter $k$. The bottom row of \Cref{fig:distance_auxiliary}  is for objective functions, i.e., $\left\|\h x^{(k)}\right\|_1/\left\|\h x^{(k)}\right\|_2$ and $\left\|\nabla u^{(k)}\right\|_1/\left\|\nabla u^{(k)}\right\|_2$ for  $L_1/L_2$ and  $L_1/L_2$-grad, respectively. All the plots in \Cref{fig:distance_auxiliary} decrease rapidly with respect to iteration counters, which serves as heuristic evidence of algorithmic convergence. 
%In addition, we often observe both algorithms stop before the maximum number of iterations, which again supports the convergence. 
On the other hand, the objective functions in \Cref{fig:distance_auxiliary}  look oscillatory. This phenomenon implies  difficulties in theoretically  proving the convergence, as one key step in the convergence  proof requires to show that  objective function decreases monotonically \cite{bolte2014proximal,wang2019global}. 

\begin{figure}
	\begin{center}
		\subfloat[Residual errors in $L1/L2$]{\label{fig:re_l1dl2}\includegraphics[height=0.42\textwidth]{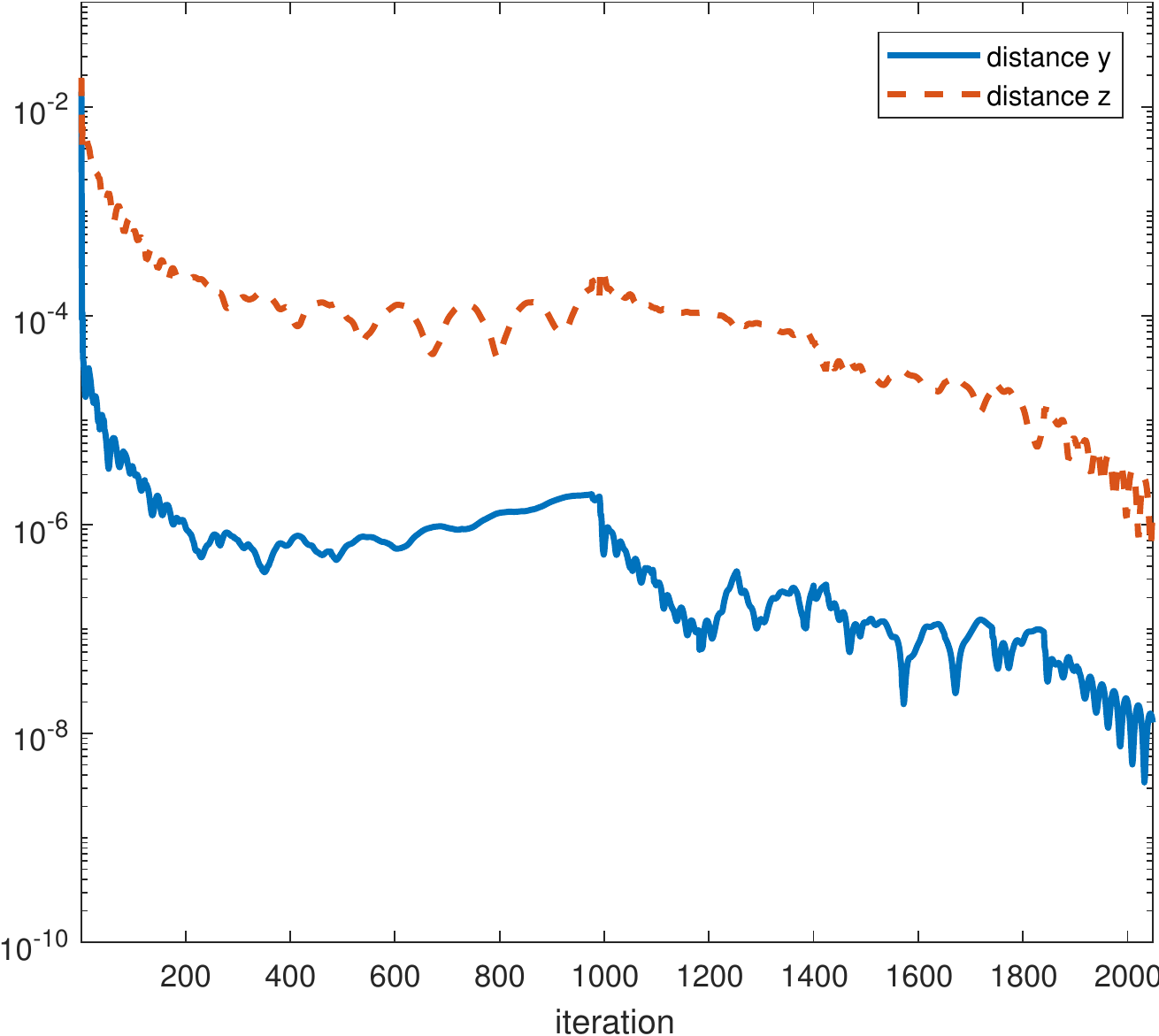}} 
\subfloat[Residual errors in $L1/L2$-grad]{\label{fig:re_l1dl2_g} \vspace{4mm} \includegraphics[height=0.42\textwidth]{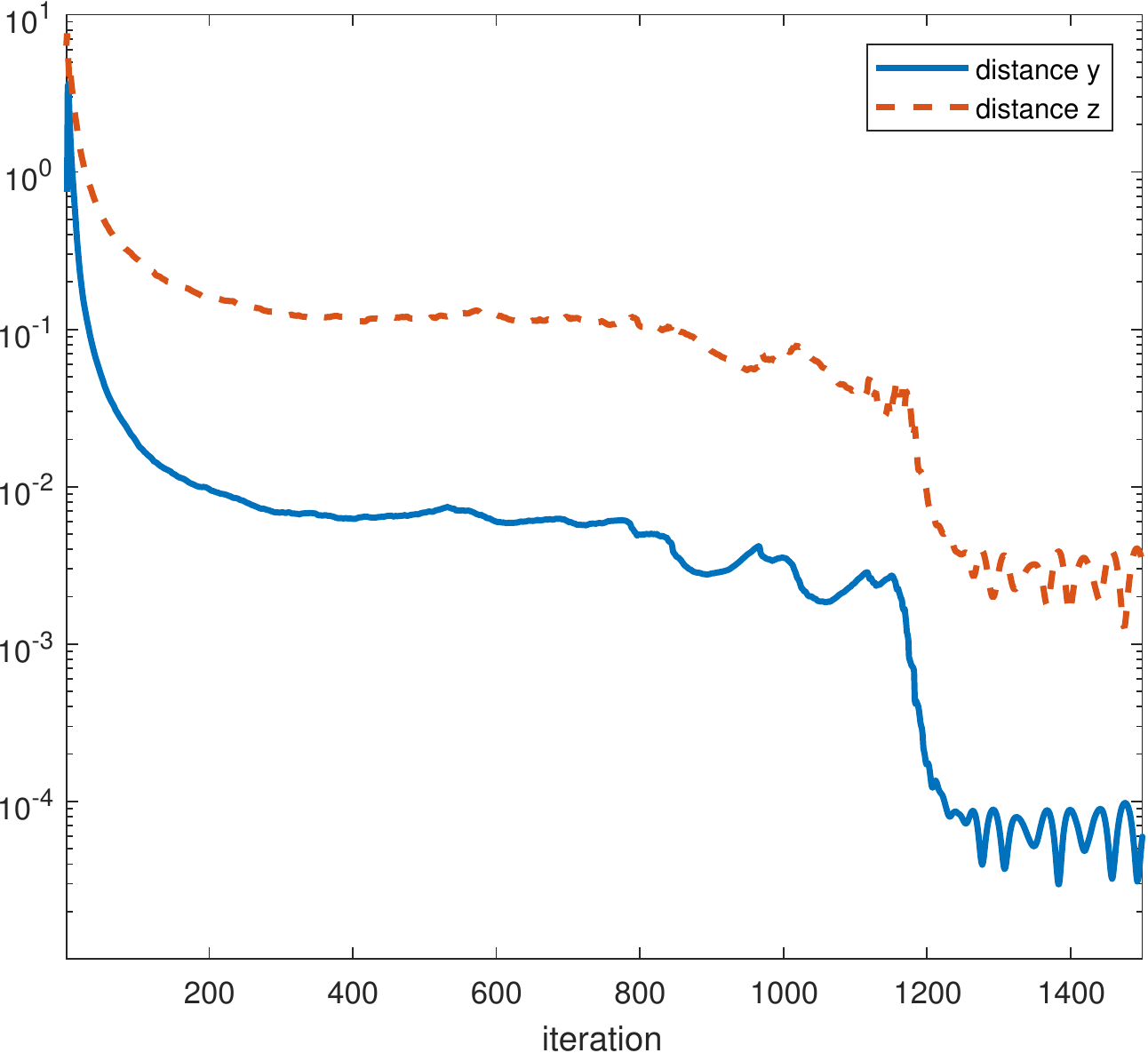}} \\
\subfloat[Objective functions of  $L1/L2$]{\label{fig:ob_l1dl2}\includegraphics[height=0.42\textwidth]{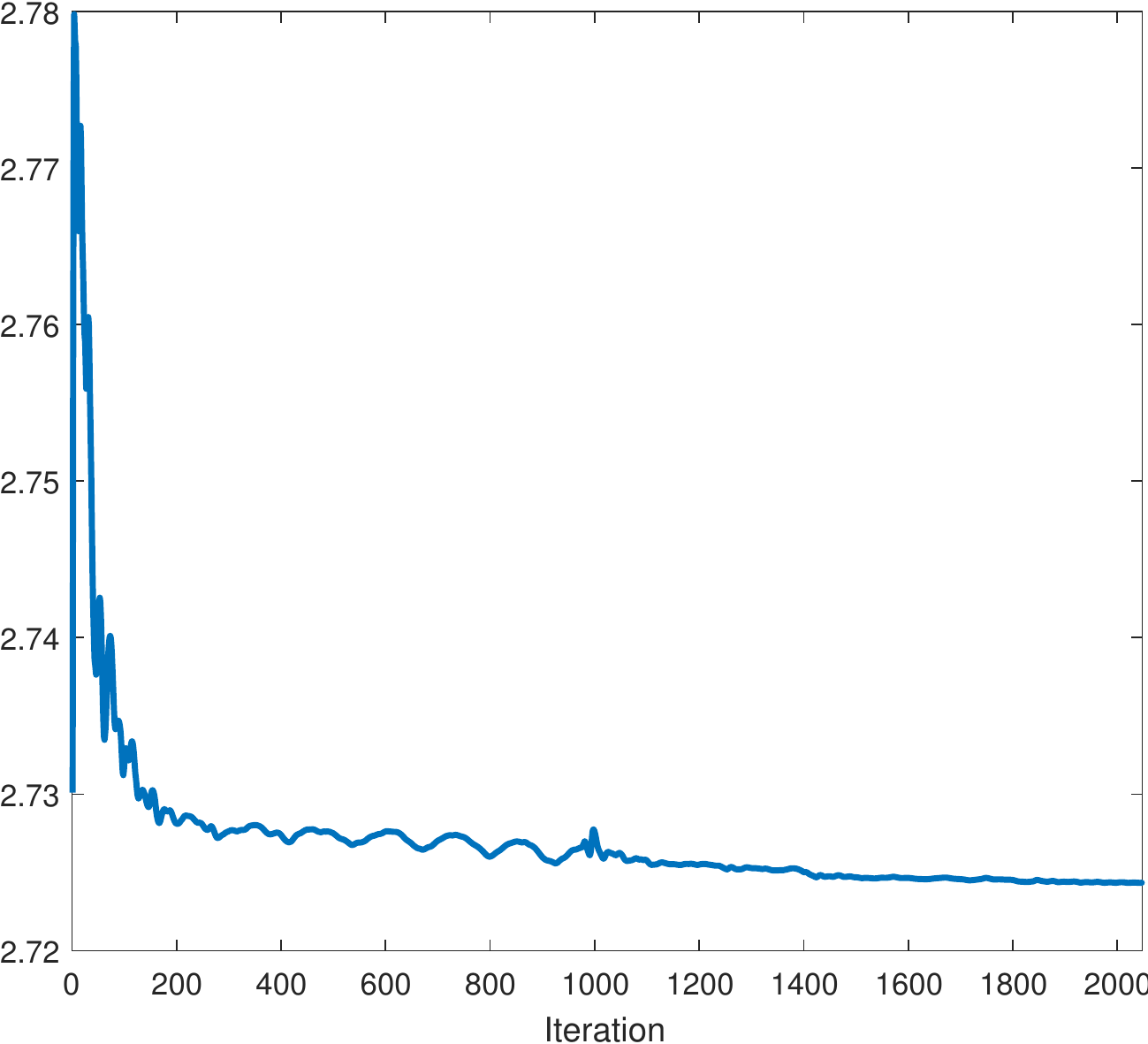}} \vspace{4mm}
\subfloat[Objective functions of $L1/L2$-grad]{\label{fig:ob_l1dl2_g}\includegraphics[height=0.42\textwidth]{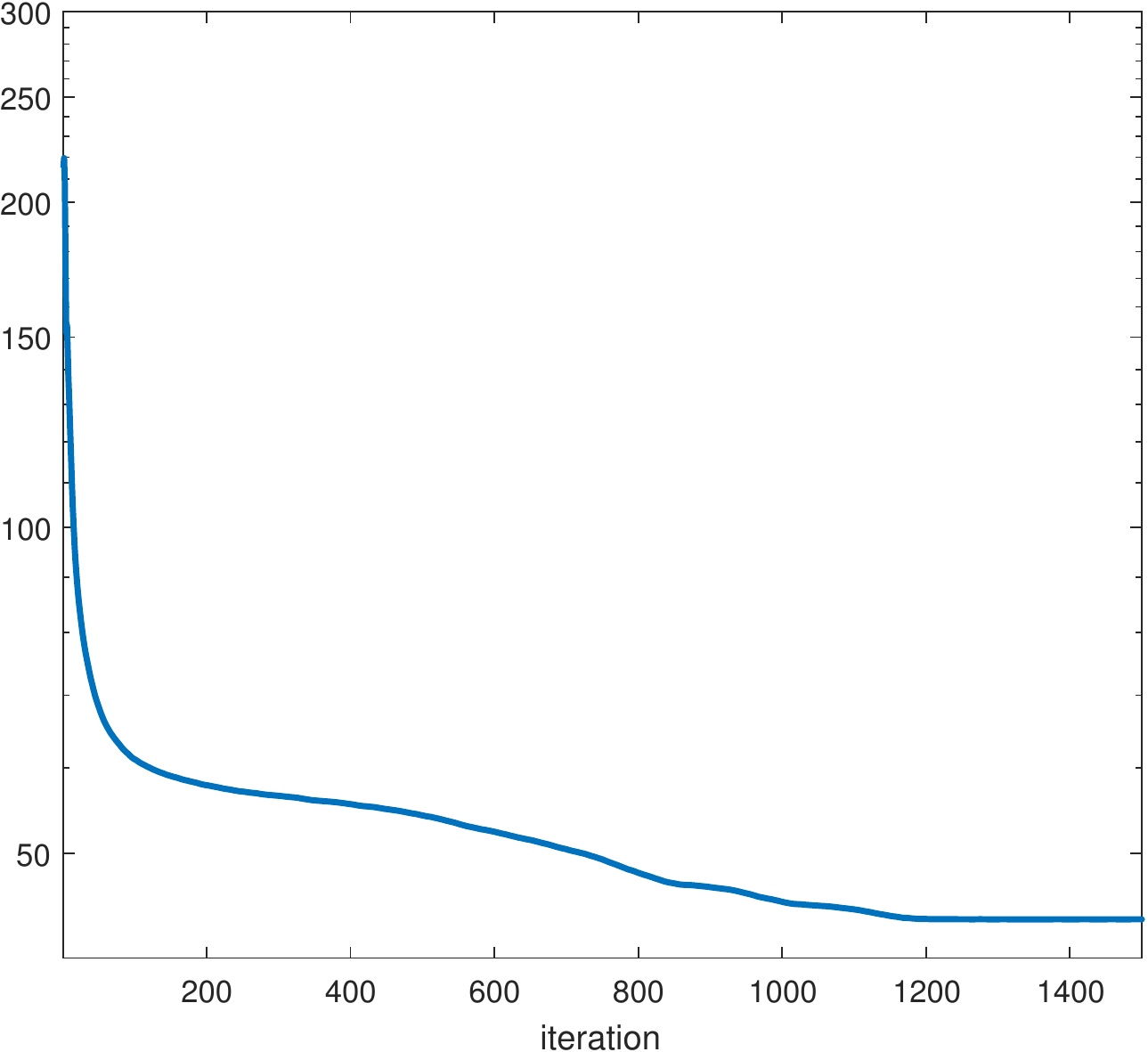}} 
		\caption{Plots  of residual errors   and objective functions for  empirically demonstrating the convergence of the proposed algorithms - $L_1/L_2$ in signal processing and $L_1/L_2$-grad with a box constraint for MRI reconstruction. }
	\end{center}
	\label{fig:distance_auxiliary}
\end{figure}

\subsection{Comparison on various models}

We now compare the proposed $L_1/L_2$ approach with  other sparse recovery models: $L_1$, $L_p$ \cite{chartrand07},  $L_1$-$L_2$ \cite{yinLHX14,louYHX14}, and TL1 \cite{zhangX18}. We choose $p=0.5$ for $L_p$  and $a= 1$ for TL1.  The initial guess for all  the algorithms is the solution of the $
L_1$ model.  Both $L_1$-$L_2$ and TL1 are solved via the DCA, with the same stopping criterion as $L_1/L_2$, i.e., $\frac{\left\|\h x^{(k)}-  \h x^{(k-1)}\right\|_2}{\left\|\h x^{(k)}\right\|_2} \leq 10^{-8}$. As for $L_p$, we 
follow the default setting in \cite{chartrand07}.

We evaluate the performance of  sparse recovery in terms of \textit{success rate}, defined as the number of successful trials over the total number of trials. A  success is declared if the relative error of the reconstructed solution $\h x^\ast$ to the ground truth $\h x$ is less than $10^{-3}$,  i.e.,  
$
\frac{\|\h x^\ast-  \h x\|_2}{\|\h x\|_2} \leq 10^{-3}. 
$
We further categorize the failure of not recovering the ground-truth as \textit{model failure} and \textit{algorithm failure}. In particular, we compare the objective function $\Fcal(\cdot)$ at the ground-truth $\h{x}$ and at the reconstructed solution $\h{x}^\ast$. If $\Fcal (\h{x}) > \Fcal (\h{x}^\ast)$, it means that $\h{x}$ is not a global minimizer of the model, in which case we call \textit{model failure}. On the other hand, $\Fcal (\h{x})<\Fcal (\h x^\ast)$ implies that the algorithm does not reach a global minimizer, which is referred to as \textit{algorithm failure}. Although this type of analysis is not deterministic, it sheds some lights on which direction to improve: model or algorithm. For example, it was reported in \cite{maLH17} that $L_1$ has the highest model-failure rates, which justifies the need for nonconvex models.

\begin{figure}[htbp]
	\centering 
	\subfloat[Success rates ($F=5$)]{\label{fig:SR_F5}\includegraphics[width=0.49\textwidth]{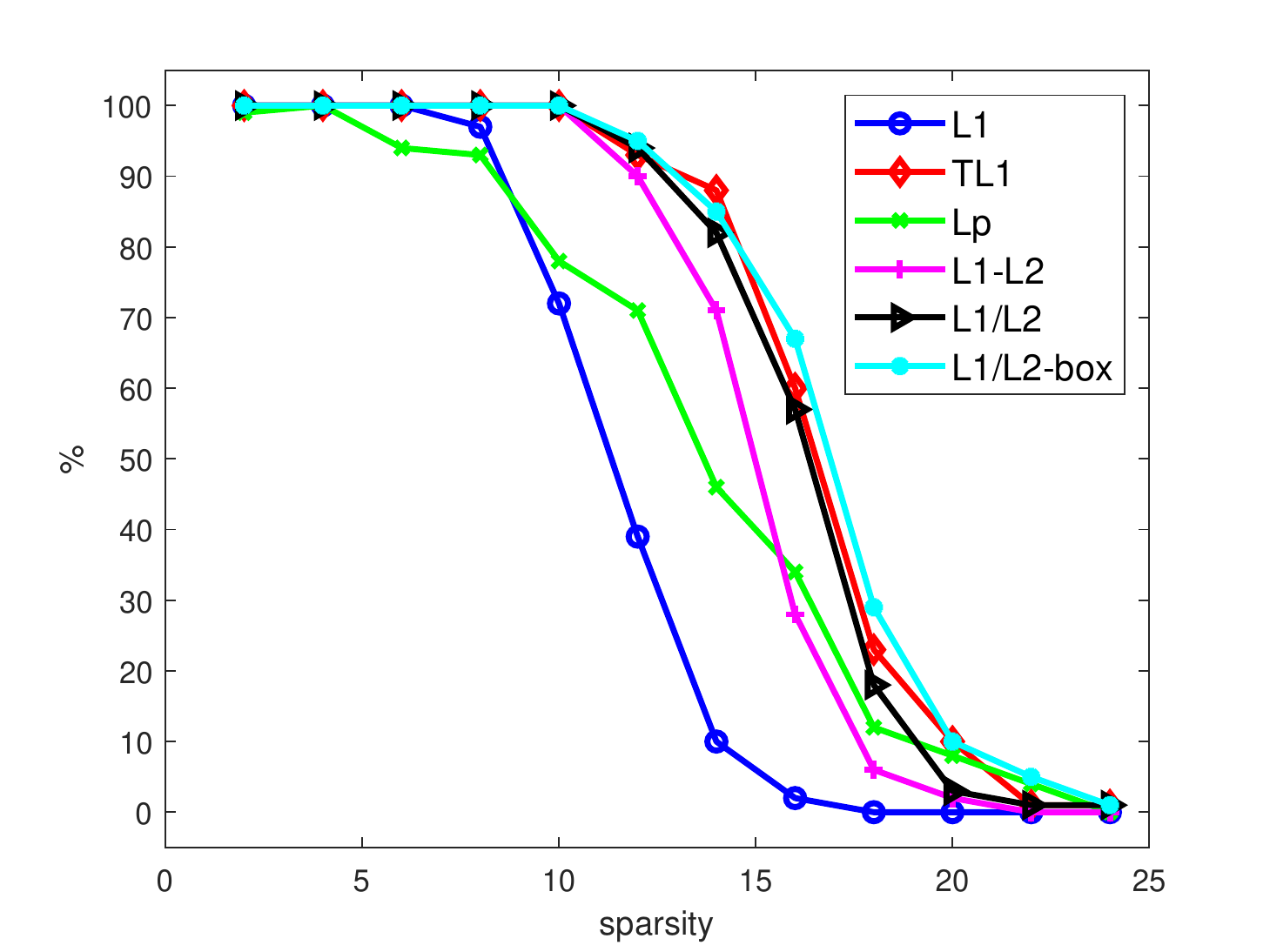}} 
	\subfloat[Success rates ($F=10$)]{\label{fig:SR_F10}\includegraphics[width=0.49\textwidth]{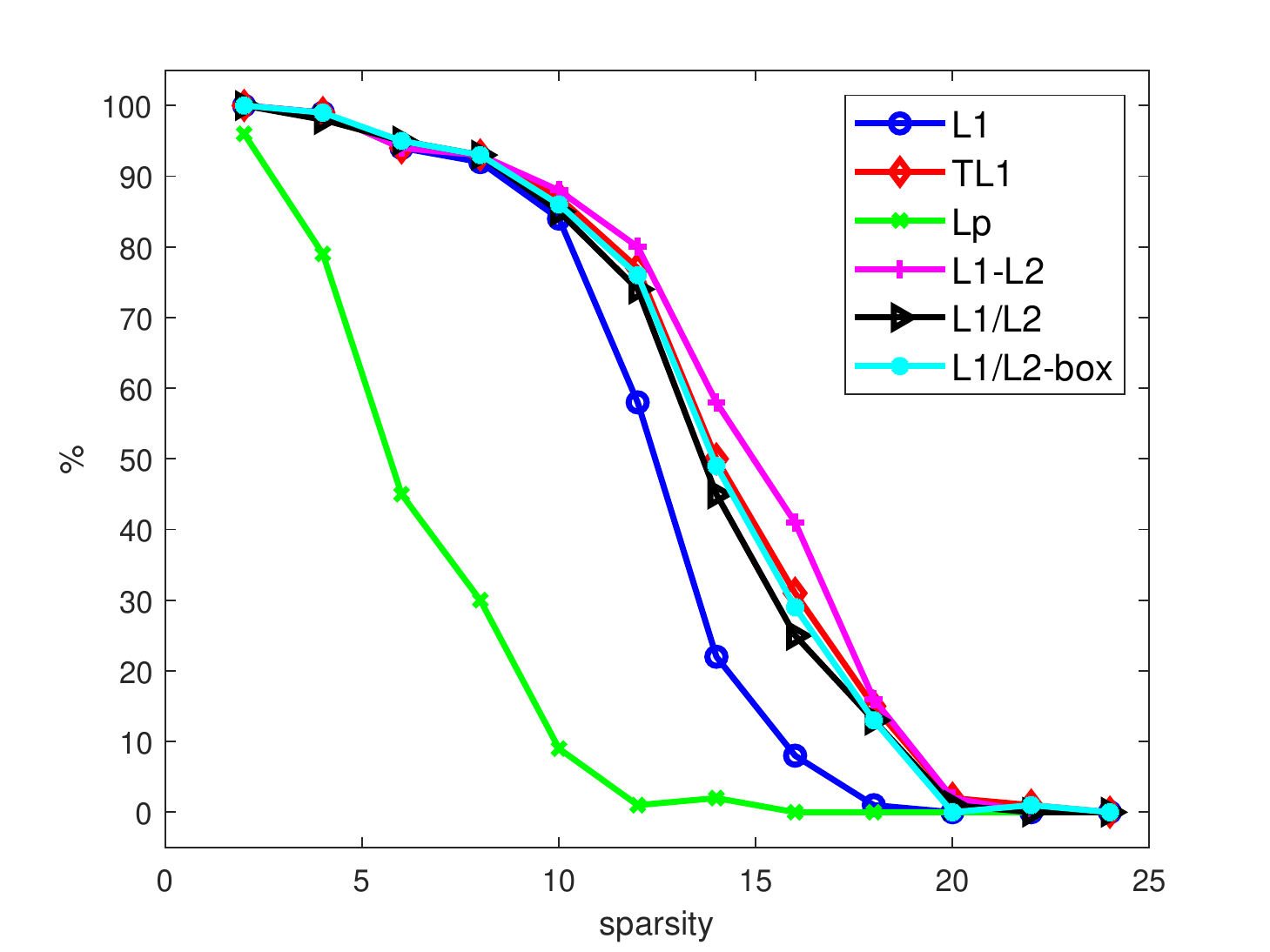}}\\
	\subfloat[Model failures ($F=5$)]{\label{fig:MF_F5}\includegraphics[width=0.49\textwidth]{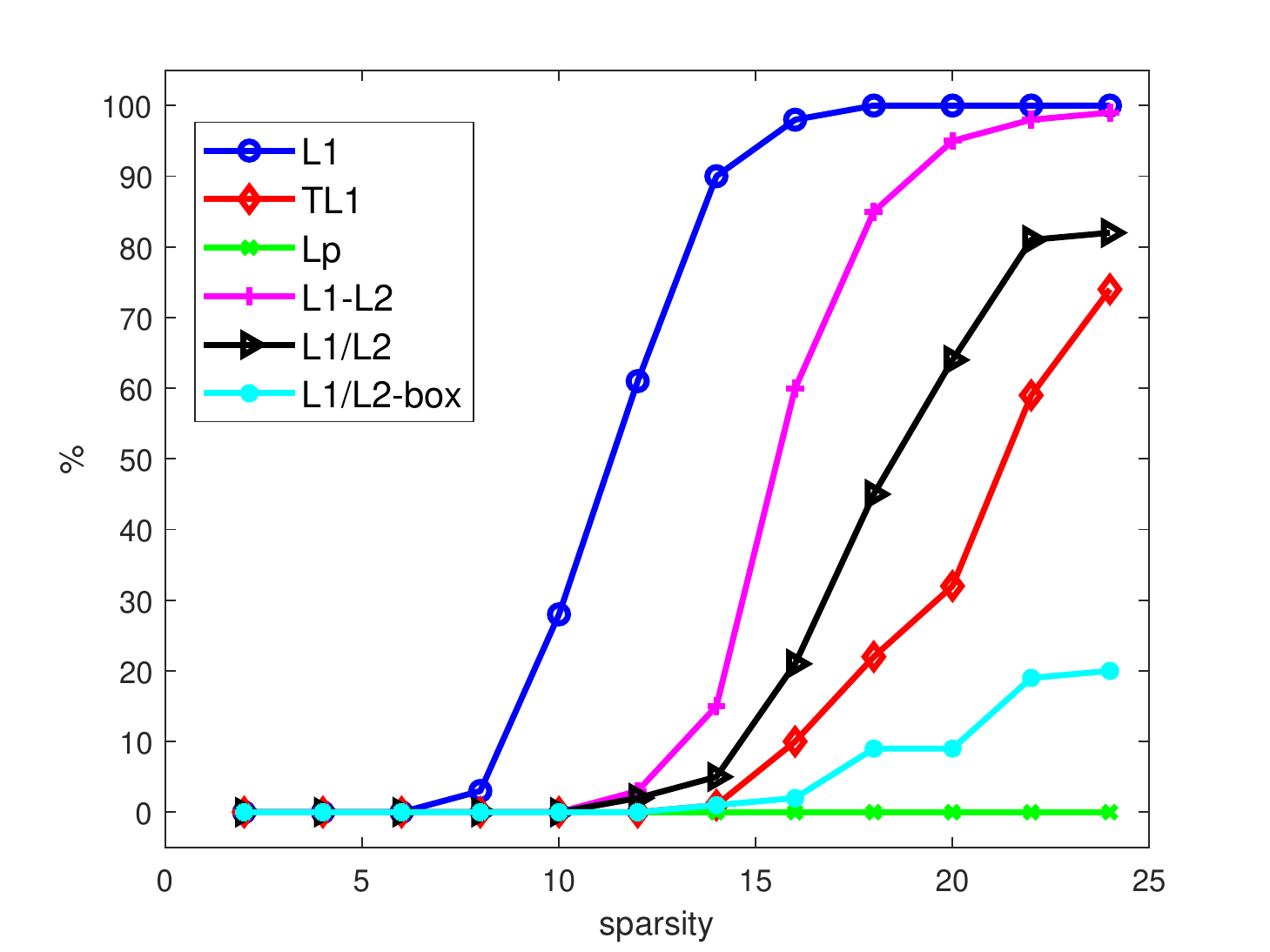}} 
	\subfloat[Model failures ($F=10$)]{\label{fig:MF_F10}\includegraphics[width=0.49\textwidth]{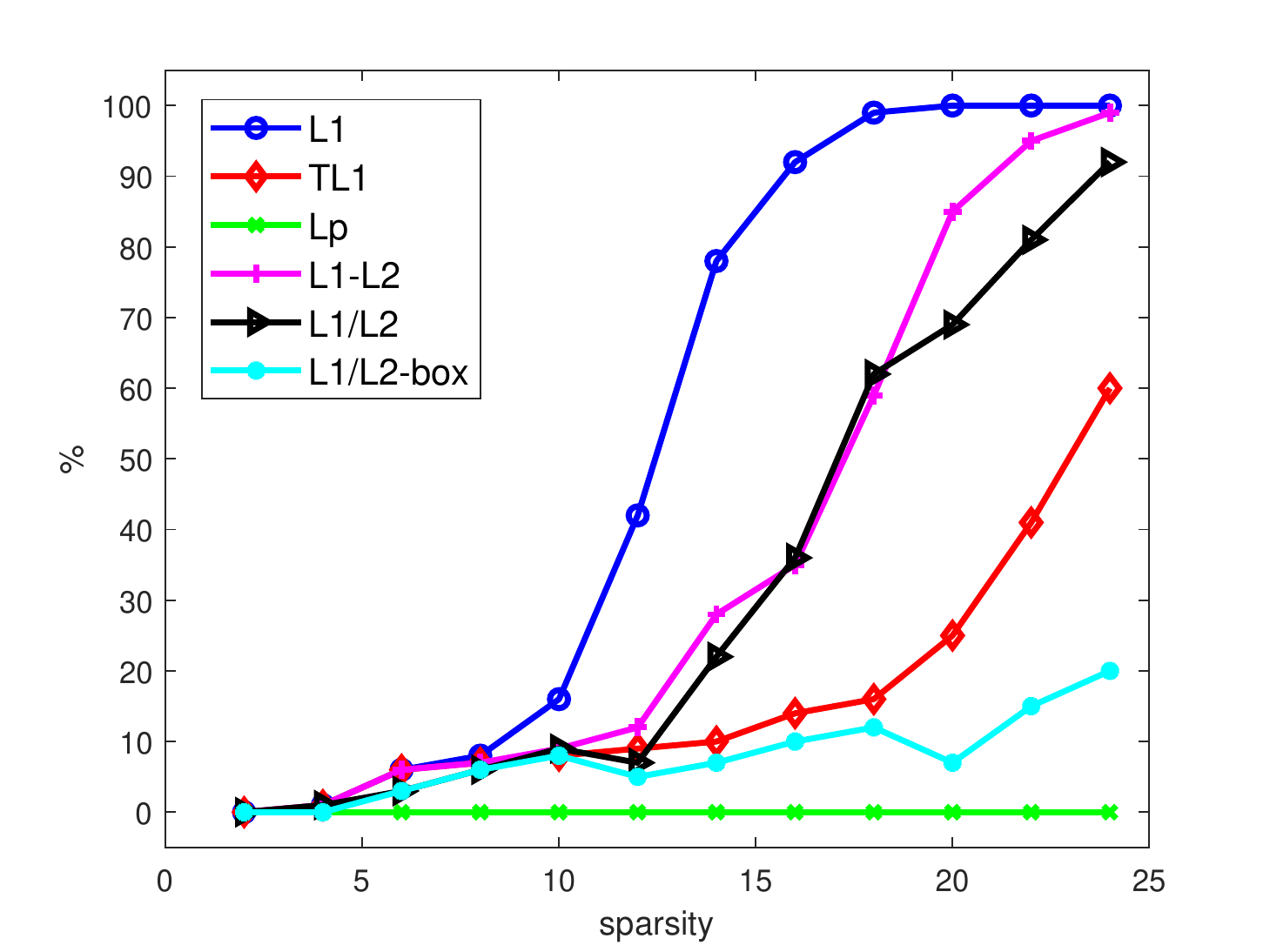}}\\
	\subfloat[Algorithm failures ($F=5$)]{\label{fig:AF_F5}\includegraphics[width=0.49\textwidth]{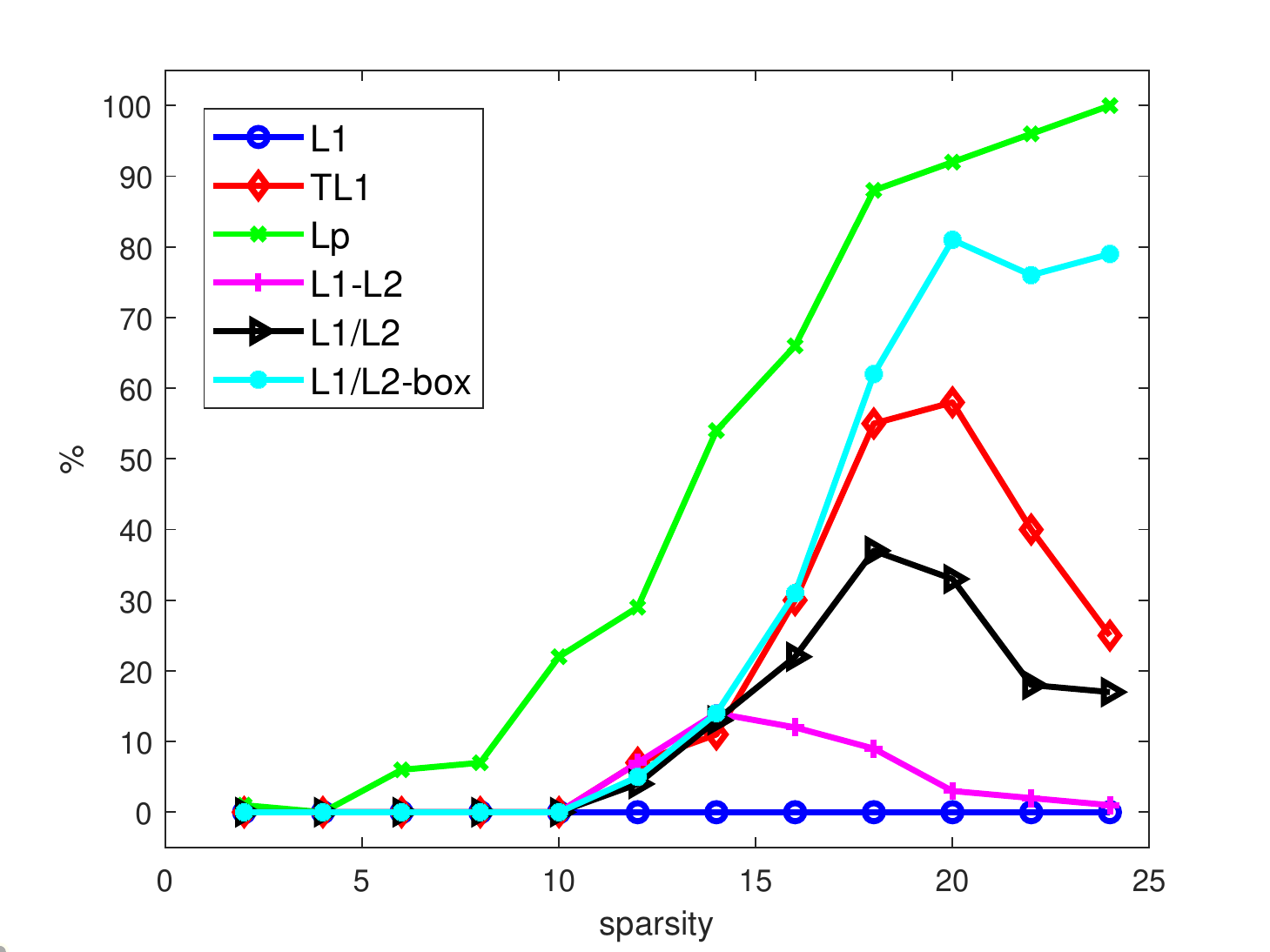}} 
	\subfloat[Algorithm failures ($F=10$)]{\label{fig:AF_F10}\includegraphics[width=0.49\textwidth]{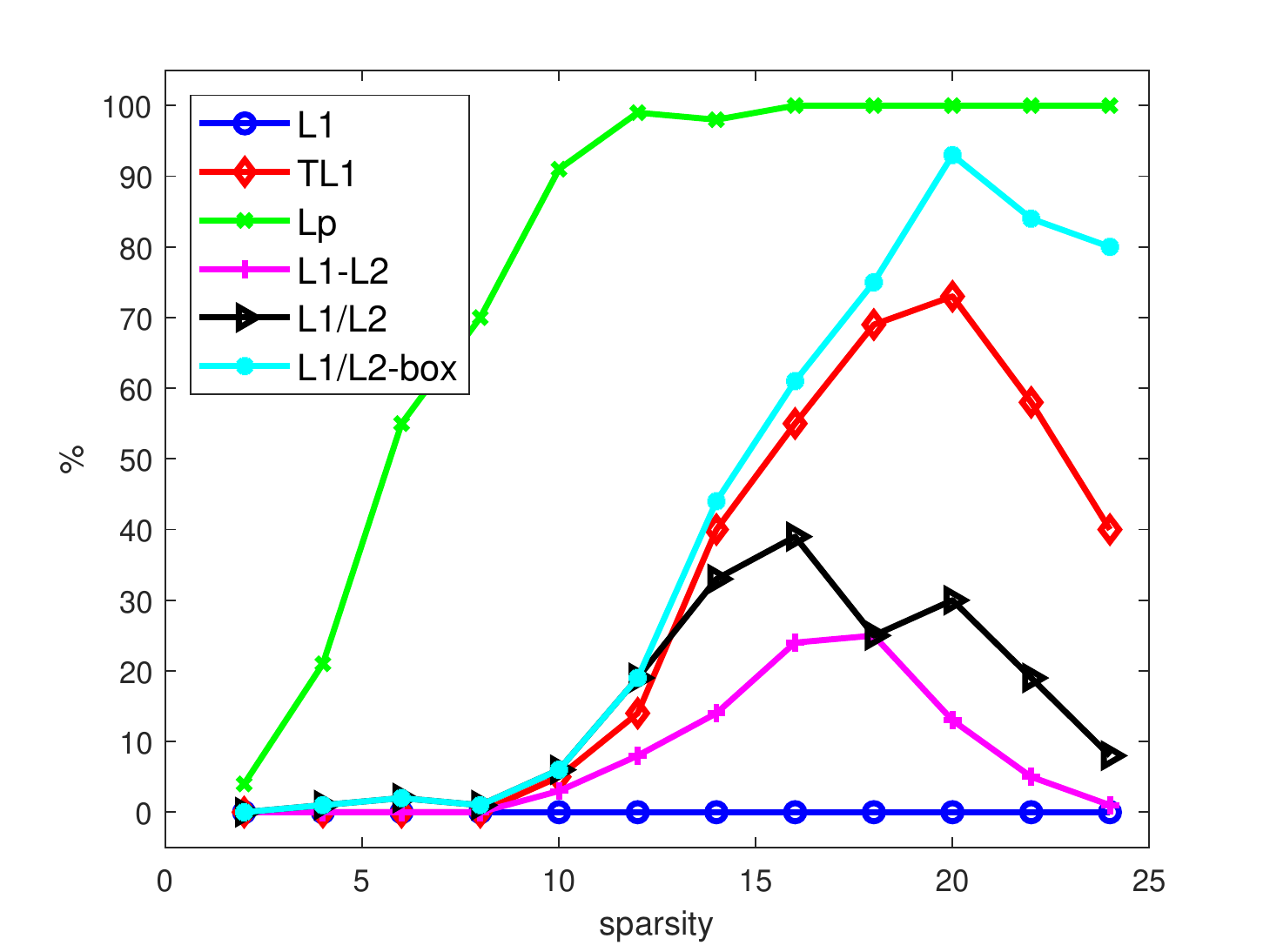}}\\
	\caption{Success rates, model failures, algorithm failures for 6 algorithms in the case of oversampled DCT matrices. }
	\label{fig:SR_MF_AF_uniform}
\end{figure}

In \Cref{fig:SR_MF_AF_uniform}, we examine two coherence levels: $F= 5$ corresponds to relatively low coherence and $F=20$ for higher coherence. The success rates of various models reveal that $L_1/L_2$-box performs the best at $F=5$ and is comparable to $L_1$-$L_2$ for the highly coherent case of $F=20$. We look at Gaussian matrix with $r=0.1$ and $r=0.8$ in \Cref{fig:SR_MF_AF_dynamic}, both of which exhibit very similar performance of various models. In particular, the $L_p$ model gives the best results for the Gaussian case, which is consistent in the literature \cite{yinEX14,louYHX14}. The proposed model of $L_1/L_2$-box is the second best for such incoherent matrices.

By comparing  $L_1/L_2$  with and without box among the plots for success rates and model failures, we can draw the conclusion that the box constraint can mitigate the inherent drawback of the $L_1/L_2$ model, thus improving the recovery rates. 
In addition, $L_1/L_2$ is the second lowest in terms of model failure rates and simply adding a box constraint also increases the occurrence of algorithm failure compared to the none box version. These two observations suggest a need to further improve upon algorithms of minimizing $L_1/L_2$.

Finally, we provide the computation time for all the competing algorithms in \Cref{tab:compare_time} with the shortest  time in each case highlighted in bold. The time for $L_1$ method is not included, as all the other methods use the $L_1$ solution as initial guess. It is shown that TL1 is the fastest for relatively lower sparsity levels and the proposed $L_1/L_2$-box is the most efficient at higher sparsity levels. The computational times for all these methods seem consistent with DCT and Gaussian matrices. 

\begin{figure}[htbp]
	\centering 
	\subfloat[Success rates ($r=0.1$)]{\label{fig:SR_r1}\includegraphics[width=0.49\textwidth]{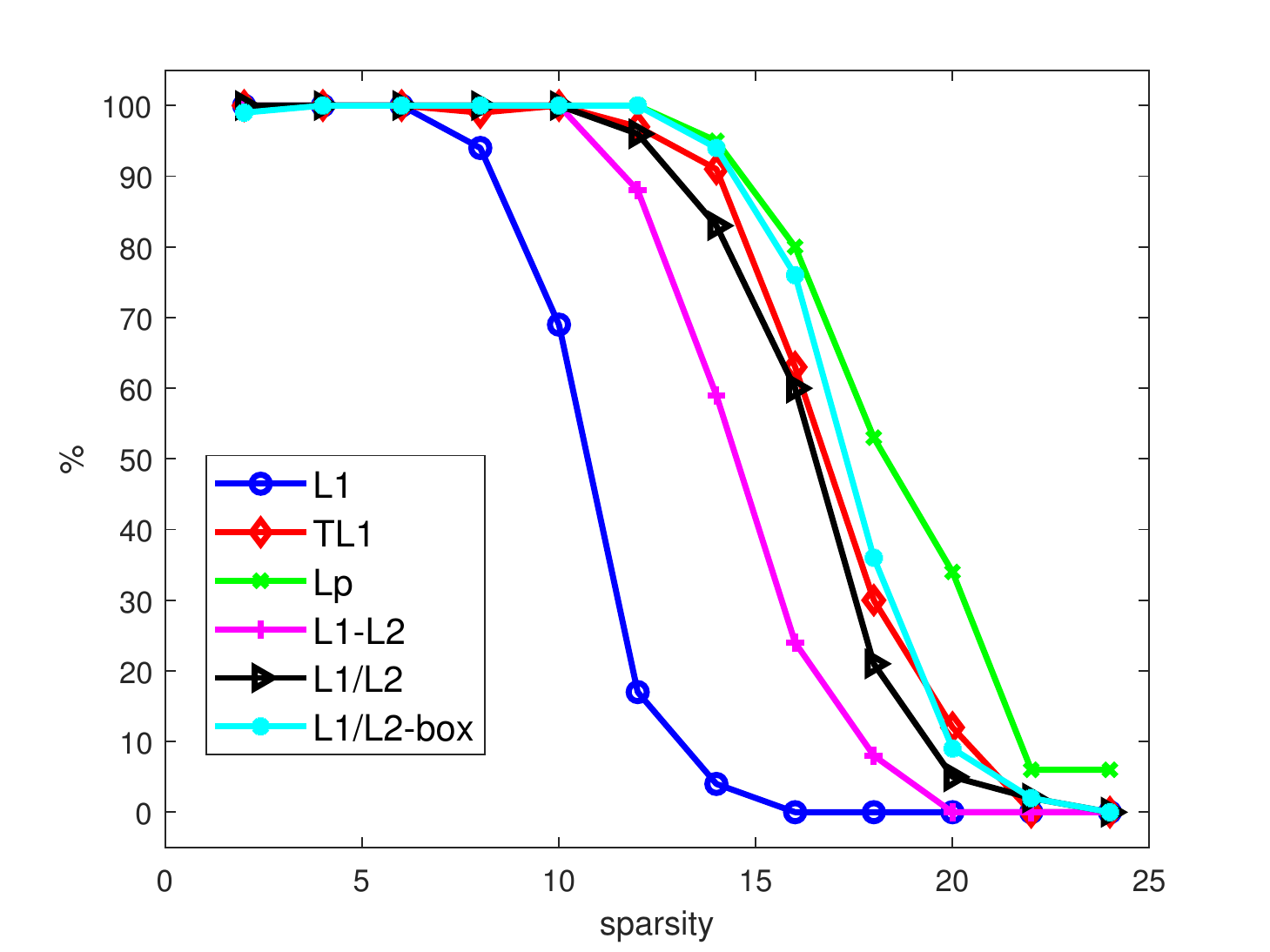}} 
	\subfloat[Success rates ($r=0.8$)]{\label{fig:SR_r8}\includegraphics[width=0.49\textwidth]{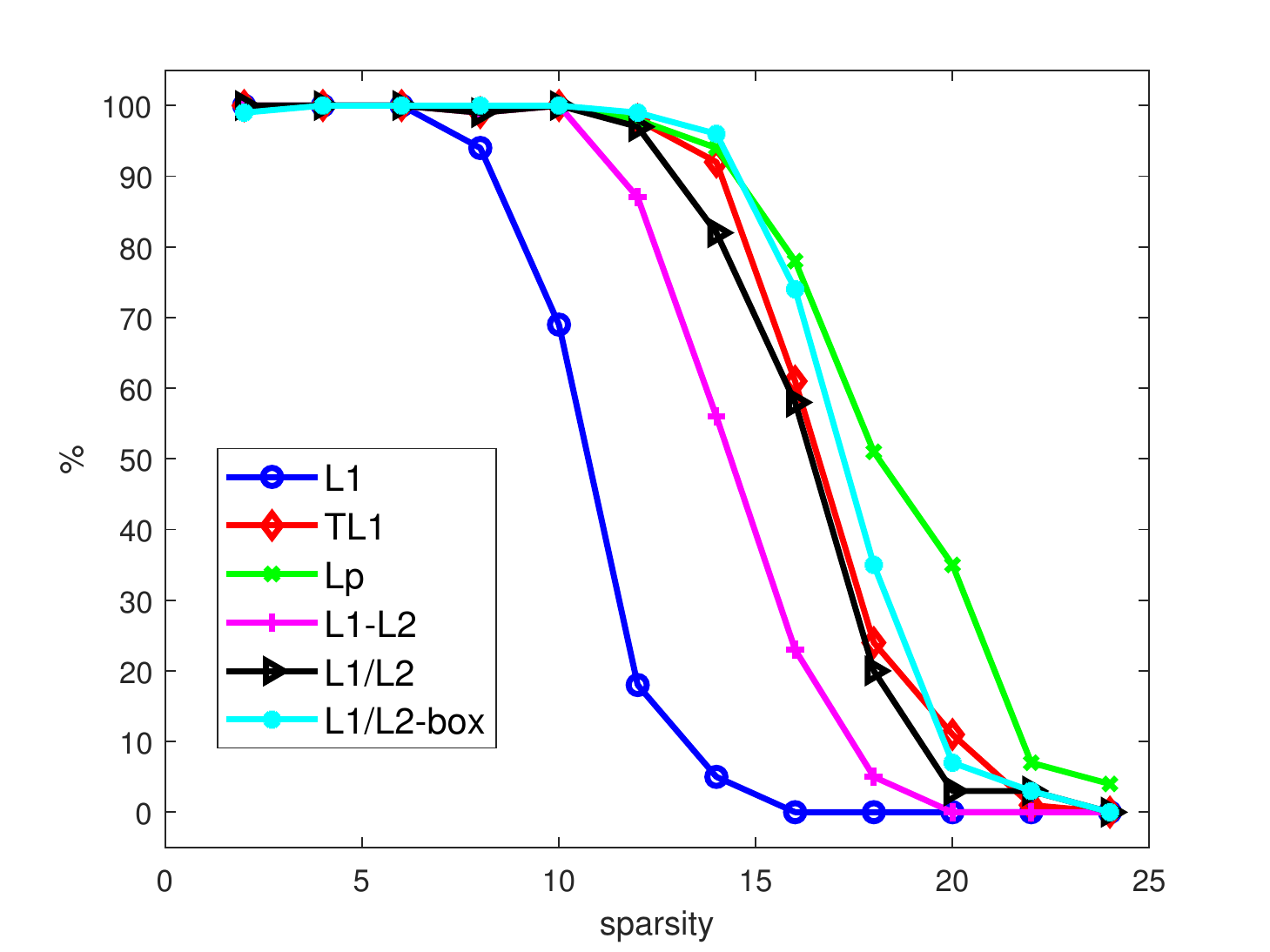}}\\
	\subfloat[Model failures ($r=0.1$)]{\label{fig:MF_r1}\includegraphics[width=0.49\textwidth]{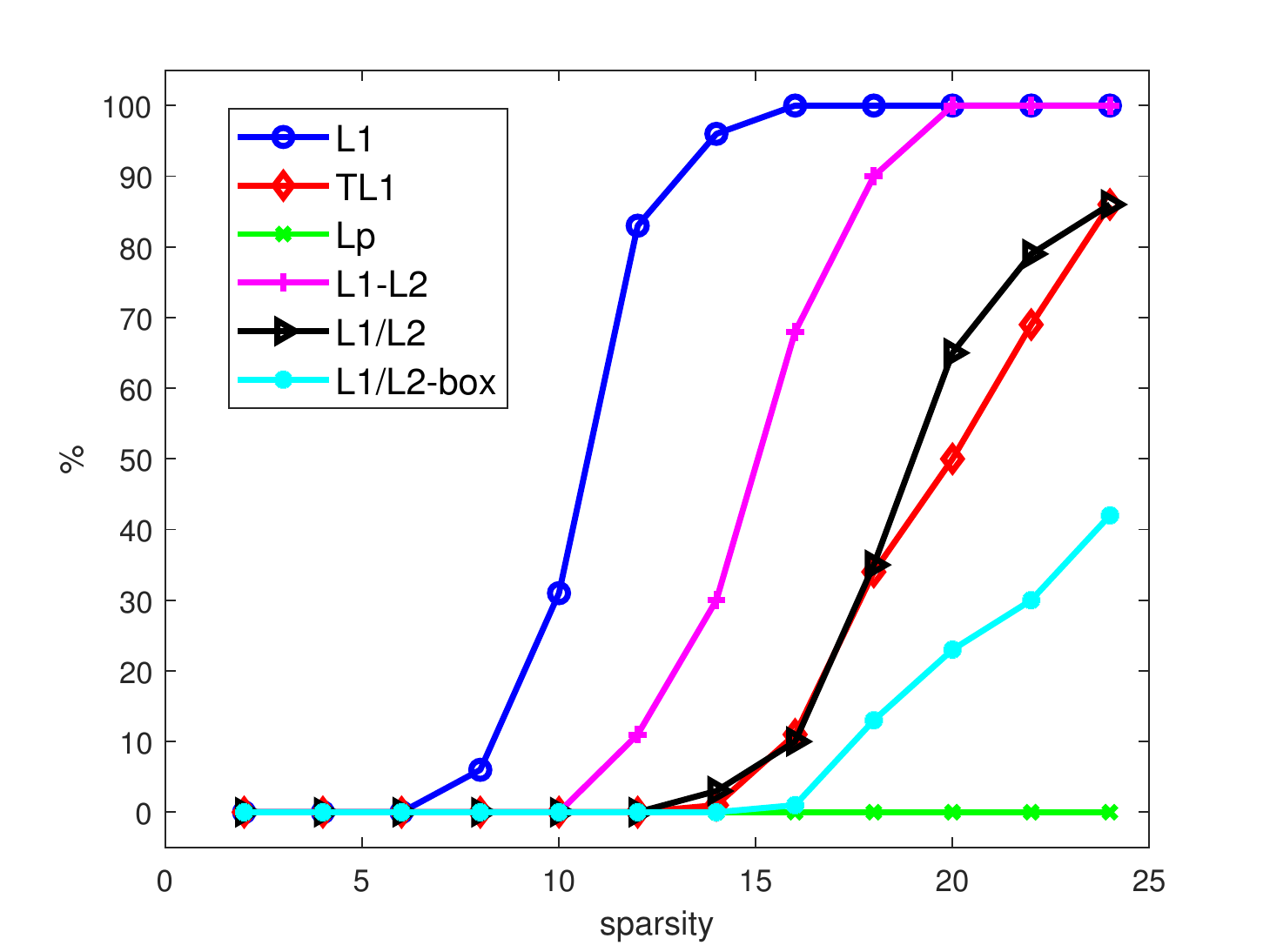}} 
	\subfloat[Model failures ($r=0.8$)]{\label{fig:MF_r8}\includegraphics[width=0.49\textwidth]{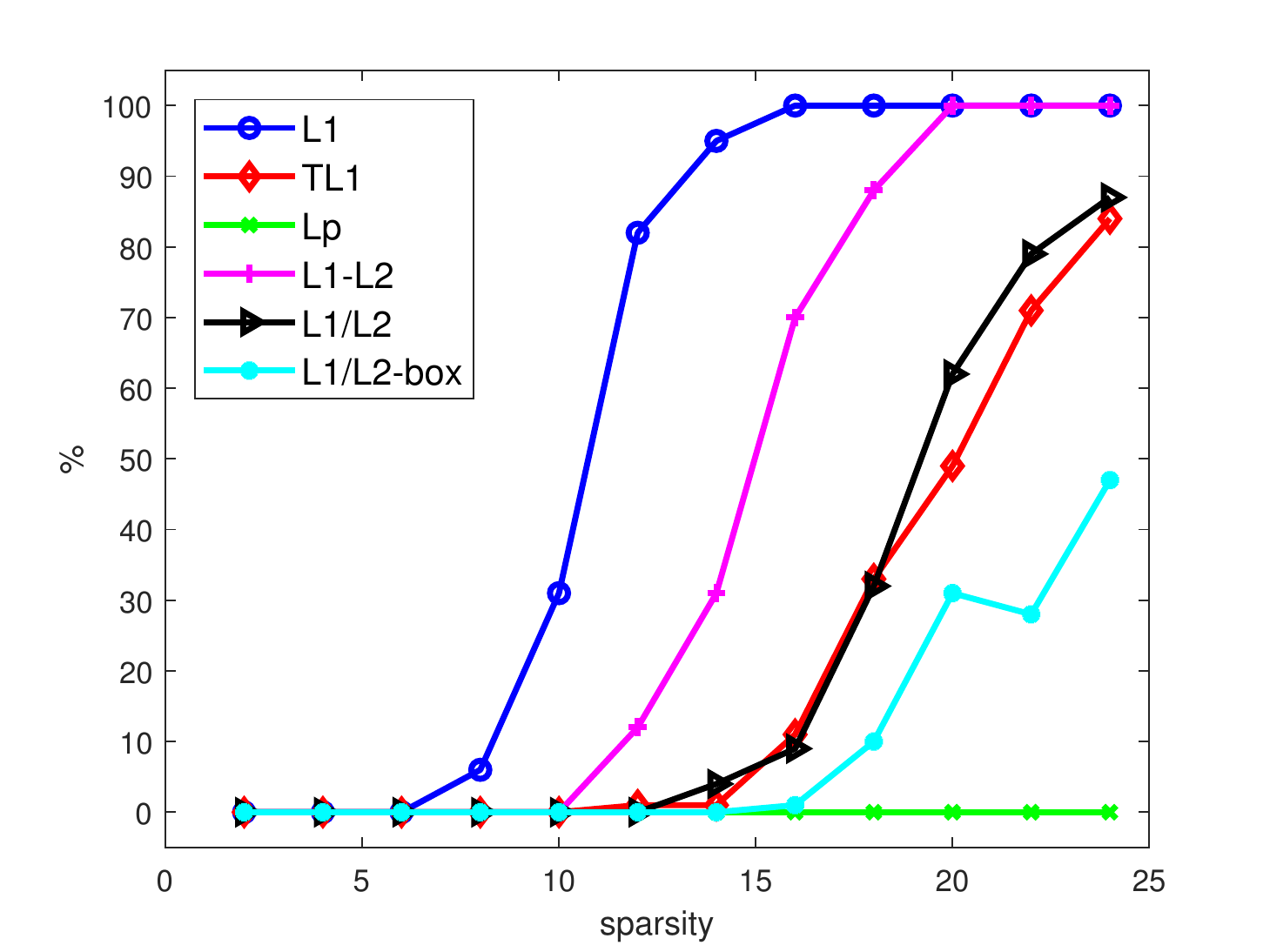}}\\
	\subfloat[Algorithm failures ($r=0.1$)]{\label{fig:AF_r1}\includegraphics[width=0.49\textwidth]{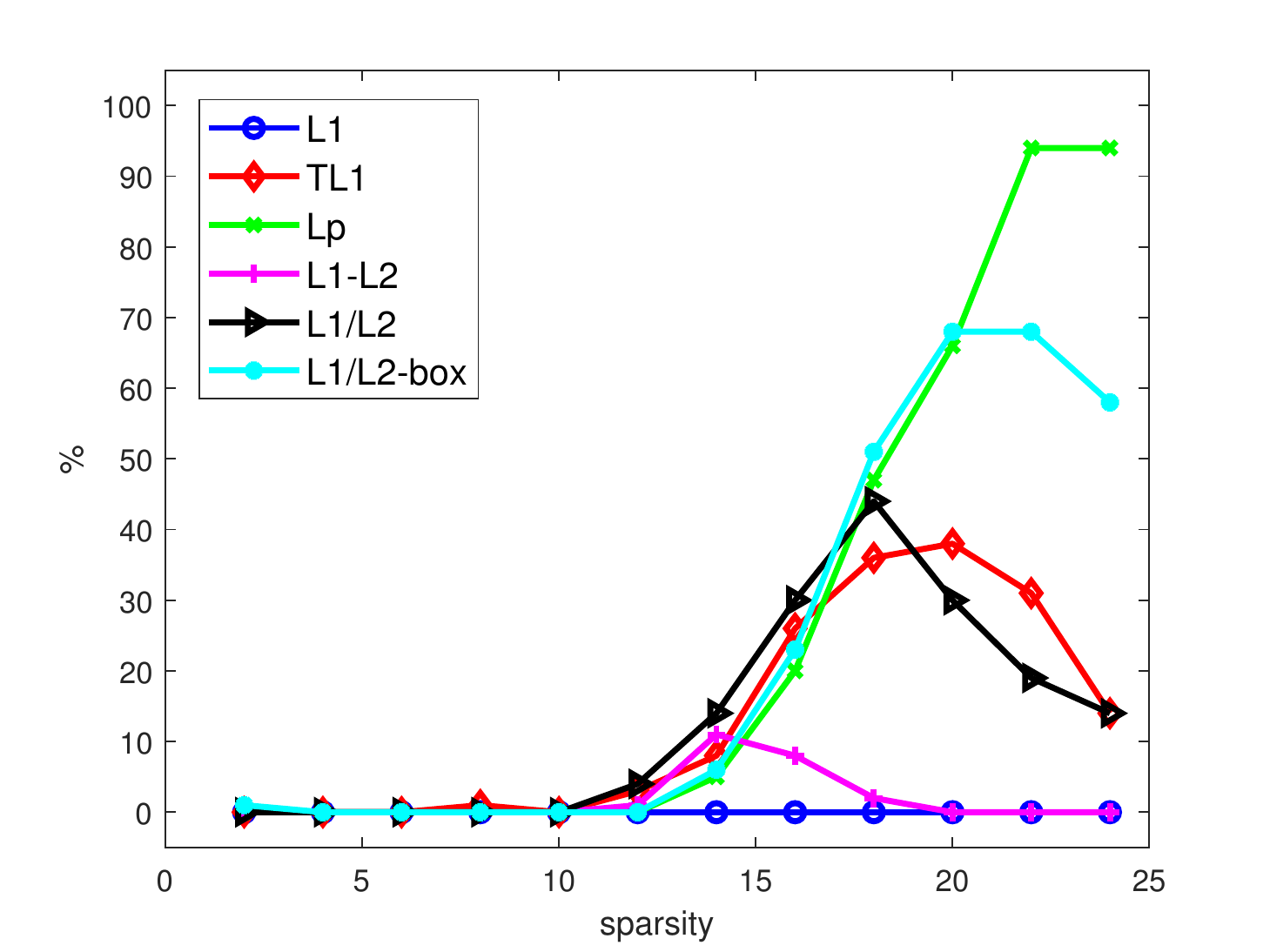}} 
	\subfloat[Algorithm failures ($r=0.8$)]{\label{fig:AF_r8}\includegraphics[width=0.49\textwidth]{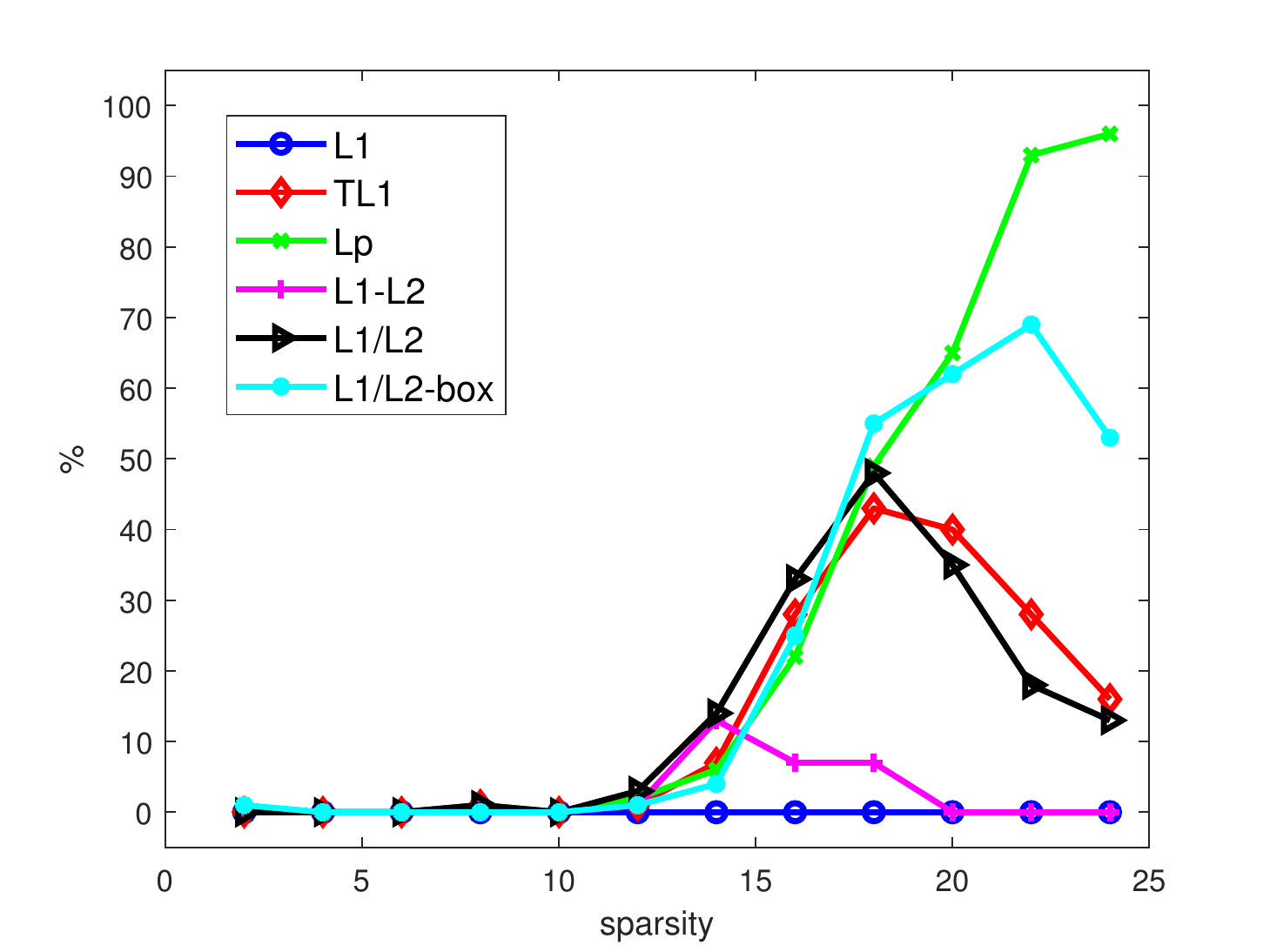}}\\
	\caption{Success rates, model failures, algorithm failures for 6 algorithms in the Gaussian matrix case. }
	\label{fig:SR_MF_AF_dynamic}
\end{figure}

%We present the results of two dynamic range cases ($D=3$ and $D=5$) in \Cref{fig:SR_MF_AF_dynamic}. The coherence factor $F$ is fixed at 10. The $L_1/L_2$ model performs the best in both two cases with highest success rates.  The $L_p$ is worst and its failures are most attributed to the algorithm. The $L_1$-$L_2$ model performs similar to the ratio of $L_1/L_2$ but with more model failures. As for TL1, it has an incremental improvement over the standard $L_1$ approach in dynamic case, which suggests its internal parameter $a$ should be adjusted according to the dynamic range. This is beyond the scope of this paper. The computation time for the dynamic case is provided in \Cref{Tab:time_dynamic}, which shows that the $L_1$-$L_2$  seems the fastest and $L_1/L_2$ is comparable.

\begin{table}[tbhp]
	{\footnotesize
		\captionsetup{position=top} %<- Needed for using subtables created with the subfig package
		\caption{Computation time (sec.) in 5 algorithms.}\label{tab:compare_time} \begin{center}
			\subfloat[DCT matrix]{\label{Tab:time_uniform} 
				\begin{tabular}{|c|c|c|c|c|c|c|c|} \hline
					\multicolumn{8}{|c|}{$F=5$} \\ \hline
					sparsity & 2  &   6  &  10  &  14  &  18  &  22 & mean \\ \hline
					%$L_1$ & 0.053 & 0.067 & 0.079 &  \bf 0.079 &  \bf 0.083 & \bf 0.081 & \bf 0.074 \\ \hline
					TL1 & \bf 0.049 & \bf 0.050 & \bf 0.066 & \bf  0.207 & 0.618 & 0.795 & 0.298 \\ \hline
					$L_p$ & 0.061 & 0.137 & 0.209 & 0.355 & 0.515 & 0.565 & 0.307\\ \hline
					$L_1$-$L_2$ & \bf 0.049 & \bf 0.050 &  0.071 & \ 0.260 & 0.550 & 0.625 & 0.267\\ \hline
					$L_1/L_2$ & 0.276 & 0.279 & 0.311 & 0.353 & 0.358 & 0.366 & 0.324 \\ \hline
					$L_1/L_2$-box & 0.102 & 0.183 & 0.247 & 0.313 & \bf 0.325 & \bf 0.332 & \bf 0.250 \\ \hline\hline
					\multicolumn{8}{|c|}{$F=10$} \\ \hline
					sparsity & 2  &   6  &  10  &  14  &  18  &  22 & mean \\ \hline
					TL1 & \bf 0.048 & \bf  0.069 & \bf  0.092 & 0.330 & 0.654 & 0.755 & 0.325 \\ \hline
					$L_p$ & 0.094 & 0.254 & 0.423 & 0.472 & 0.530 & 0.534 & 0.385\\ \hline
					$L_1$-$L_2$ &  0.049 &  0.070 & 0.093 & \bf 0.272 & 0.598 & 0.677 & 0.293\\ \hline
					$L_1/L_2$ & 0.263 & 0.272 & 0.295 &  0.340 & 0.355 & 0.356 & 0.314 \\ \hline
					$L_1/L_2$-box & 0.090 & 0.179 & 0.239 & 0.301 & \bf 0.324 & \bf 0.322 & \bf 0.243 \\ \hline					
				\end{tabular}
			}\\
			\subfloat[Gaussian matrix]{\label{Tab:time_dynamic}
				\begin{tabular}{|c|c|c|c|c|c|c|c|} \hline
					\multicolumn{8}{|c|}{$r = 0.1$} \\ \hline
					sparsity & 2  &   6  &  10  &  14  &  18  &  22 & mean \\ \hline
					TL1 & \bf 0.070 & \bf 0.069 & \bf 0.117 & 0.295 & 1.101 & 1.633 & 0.548 \\ \hline
					$L_p$ & 0.079 & 0.128 & 0.229 & \bf 0.261 & \bf 0.742 & 1.218 & \bf 0.443 \\ \hline
					$L_1$-$L_2$ & \bf 0.070 & \bf 0.069 & 0.122 & 0.399 & 0.877 & \bf 1.161 & 0.450\\ \hline
					$L_1/L_2$ & 0.864 & 0.866 & 1.175 & 1.130 & 1.210 & 1.458 & 1.117 \\ \hline
					$L_1/L_2$-box & 0.324 & 0.625 & 1.039 & 1.060 & 1.146 & 1.385 & 0.930 \\ \hline\hline
					\multicolumn{8}{|c|}{$r = 0.8$} \\ \hline
					sparsity & 2  &   6  &  10  &  14  &  18  &  22 & mean \\ \hline
					TL1 & \bf 0.050 & \bf 0.053 & \bf 0.071 & 0.239 & 0.613 & 0.750 & 0.296 \\ \hline
					$L_p$ & 0.061 & 0.094 & 0.140 & \bf 0.207 & 0.426 & 0.613 & 0.257\\ \hline
					$L_1$-$L_2$ & 0.051 & 0.054 & 0.077 & 0.306 & 0.497 & 0.576 & 0.260\\ \hline
					$L_1/L_2$ &0.277 & 0.277 & 0.324 & 0.358 & 0.364 & 0.363 & 0.327\\ \hline
					$L_1/L_2$-box & 0.102 & 0.192 & 0.265 & 0.321 & \bf 0.332 & \bf 0.327 & \bf 0.256\\ \hline	
				\end{tabular}
			}
		\end{center}
	}
\end{table}

\subsection{MRI reconstruction}

As a proof-of-concept example, we study an MRI reconstruction problem \cite{sparseMRI} to compare the performance of $L_1$, $L_1$-$L_2$, and $L_1/L_2$ on the gradient. The $L_1$ on the gradient is the celebrated TV model \cite{rudinOF92}, while $L_1$-$L_2$ on the gradient was recently proposed in \cite{louZOX15}.
We use a standard Shepp-Logan phantom as a testing image, as shown in \Cref{fig:MRI_orig}. The MRI measurements are obtained by several radical lines in the frequency domain (i.e., after taking the Fourier transform); an example of such sampling scheme using 6 lines is shown in \Cref{fig:MRI_mask}.  As this paper focuses on the constrained formulation, we do not consider noise, following the same setting as in the previous works \cite{yinLHX14,louZOX15}. Since all the competing methods ($L_1$, $L_1$-$0.5L_2$, and $L_1/L_2$) yield an exact recovery with 8 radical lines, with accuracy in the order of $10^{-8}$, we  present the  reconstructions results of  6 radical lines in 
\Cref{fig:MRI}, which illustrates that the ratio model ($L_1/L_2$) gives much better results than the difference model ($L_1$-$0.5L_2$). \Cref{fig:MRI} also includes quantitative measures of the performance by relative error (RE) between the reconstructed and ground-truth images, which shows significantly improvement of the proposed $L_1/L_2$-grad over a classic method in MRI reconstruction, called  filter-back projection (FBP), and two recent works of using $L_1$ \cite{GO} and $L_1$-$0.5L_2$ \cite{louZOX15} on the gradient. 
Note that the state-of-the-art methods in MRI reconstruction are \cite{cs:edgecs_guo2012,maLH17} that have reported exact recovery from 7 radical lines. 
%but both approaches require to tune  model parameters, while our model \eqref{eq:grad_con} is parameter free\footnote{There are three algorithmic parameters in \eqref{eq:grad_aug}, but they only affect the convergence speed of the algorithm.}.}

\begin{figure}
	\begin{center}
		\subfloat[Original]{\label{fig:MRI_orig}\includegraphics[width=0.31\textwidth]{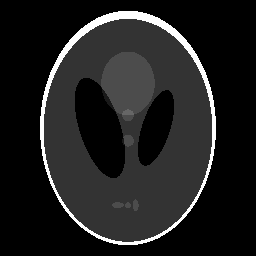}} \vspace{4mm}
		\subfloat[Sampling mask]{\label{fig:MRI_mask}\includegraphics[width=0.31\textwidth]{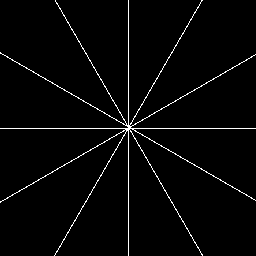}} \vspace{4mm}
		\subfloat[FBP (RE = 99.80\%)]{\label{fig:MRI_FBP}\includegraphics[width=0.31\textwidth]{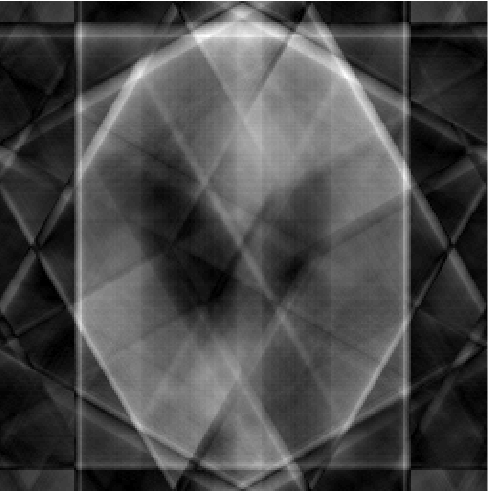}} \\
		\subfloat[$L_1$ (RE = 39.42\%)]{\label{fig:MRI_L1}\includegraphics[width=0.31\textwidth]{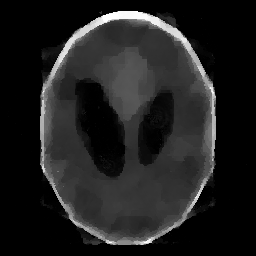}} \vspace{4mm}
		\subfloat[$L_1$-0.5$L_2$ (RE = 38.43\%)]{\label{fig:MRI_L1L2}\includegraphics[width=0.31\textwidth]{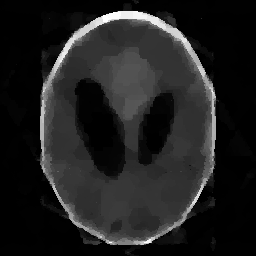}} \vspace{4mm}
		\subfloat[$L_1/L_2$ (RE = 0.04\%)]{\label{fig:MRI_L1dL2}\includegraphics[width=0.31\textwidth]{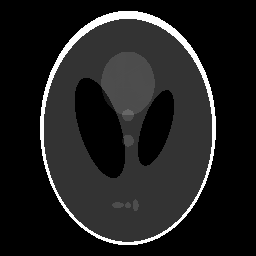}}  
	\end{center}
	\caption{MRI reconstruction results from 6 radical lines in the frequency domain (2.57$\%$ measurements). The relative errors (RE) are provided for each method. }\label{fig:MRI}
\end{figure}

\section{Empirical validations}\label{sect:ratio}

A review article \cite{candes2008introduction} indicated that two principles in CS are \textit{sparsity} and \textit{incoherence}, leading an impression that a sensing matrix with smaller coherence  is easier for sparse recovery.
However, we observe through numerical results \cite{louOX15} (also given in \Cref{fig:SR_ms}) that a more coherent matrix  gives higher recovery rates. This contradiction motivates us to collect empirical evidence regarding to  either prove or refuse whether coherence is relevant to sparse recovery.
Here we  examine one such evidence by minimizing the ratio of $L_1$ and $L_2$, which gives an upper bound for a sufficient condition of $L_1$ exact recovery, see \eqref{eq:NSPratio}.  To avoid the trivial solution of $\h x=\h 0$ to the problem of $\min\limits_{\h x}\left\{\frac{\|\h x\|_1}{\|\h x\|_2}: A\h x = \h 0 \right\}$, we incorporate a sum-to-one constraint. In other word, we define an expanded matrix $\tilde{A} = [A; \mbox{ones}(n,1)]$ (following Matlab's notation) and an expanded vector $\tilde{\h b}= [\h 0; 1]$. We then adapt the proposed method to solve for 
$\min\limits_{\h x}\left\{\frac{\|\h x\|_1}{\|\h x\|_2}: \tilde A\h x = \tilde{\h b} \right\}.$ In \Cref{fig:mean_ratio}, 
we  plot the mean value of ratios from  50  random realizations of matrices $A$ at each coherence level (controlled by $F$),  which shows that the ratio actually decreases\footnote{We also observe that the ratio stagnates for larger $F$, which is probably because of instability of the proposed method when matrix becomes more coherent.} with respect to $F$.  As the $L_0$ norm is bounded by the ratio \eqref{eq:NSPratio}, smaller ratio indicates it is more difficult to recover the signals.  Therefore,  \Cref{fig:mean_ratio} is consistent with the common belief in CS.

We postulate that an underlying reason of more coherent matrices giving better results is minimum separation (MS), as formally introduced in \cite{candes2014towards}. In \Cref{fig:SR_ms}, we enforce the minimum separation of two neighboring spikes to be 40, following the suggestion of $2F$ in \cite{DCT2012coherence} (we consider $F$ up to 20). In comparison, we also give the success rates of the $L_1$ recovery without any restrictions on MS in \Cref{fig:SR_nonms}. Note that we use the exactly same matrices in both cases (with and without MS). 
\Cref{fig:SR_nonms} does not have a clear pattern regarding how coherence affects the exact recovery, which supports our hypothesis that minimum separation plays an important role in sparse recovery. It will be our future work to analyze it throughly.

\begin{figure}
	\centering 
	\subfloat[The mean of ratios]{\label{fig:mean_ratio}	\includegraphics[width=0.45\textwidth]{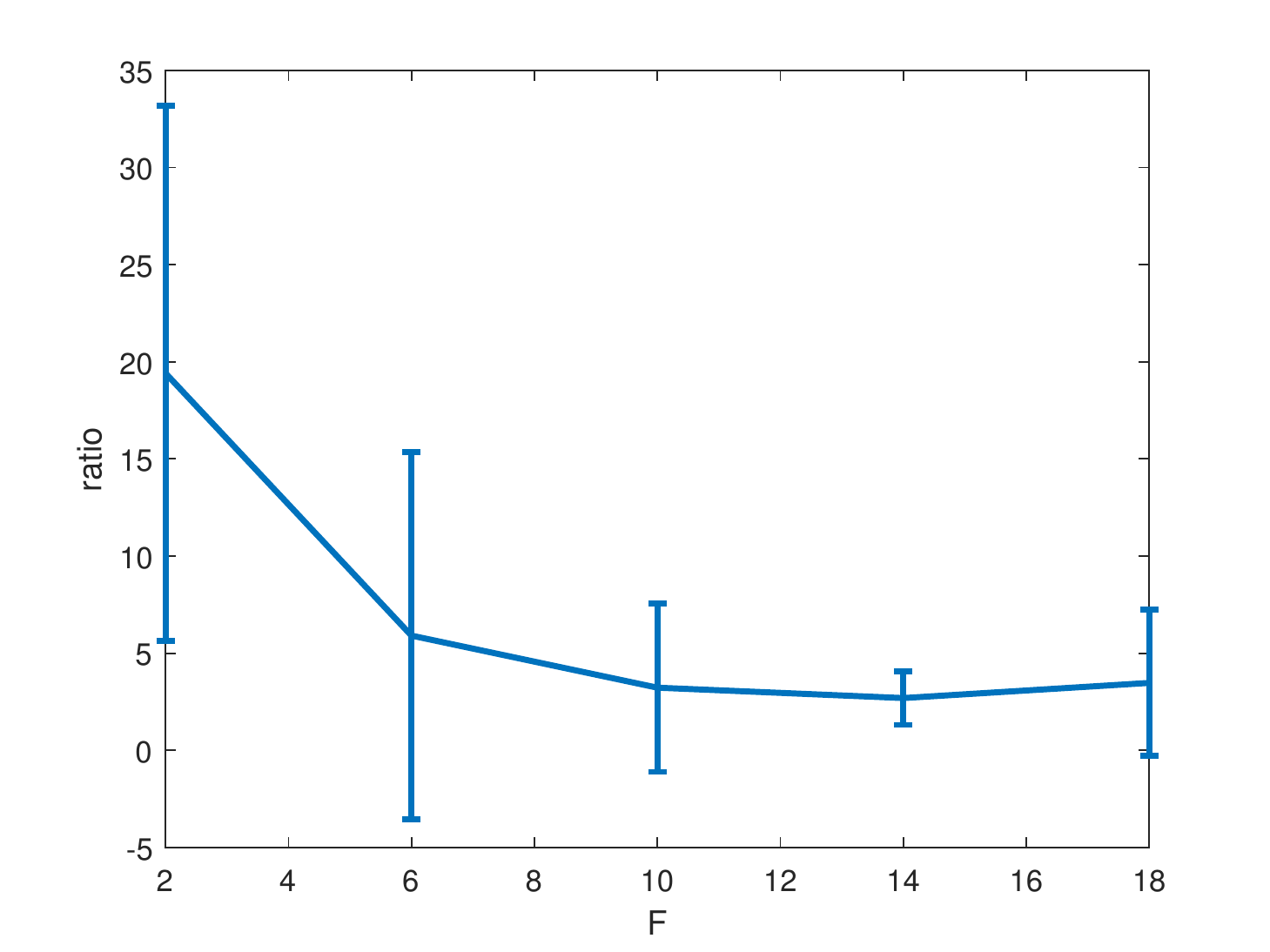}}\\
	\subfloat[Success rates with MS]{\label{fig:SR_ms}\includegraphics[width=0.45\textwidth]{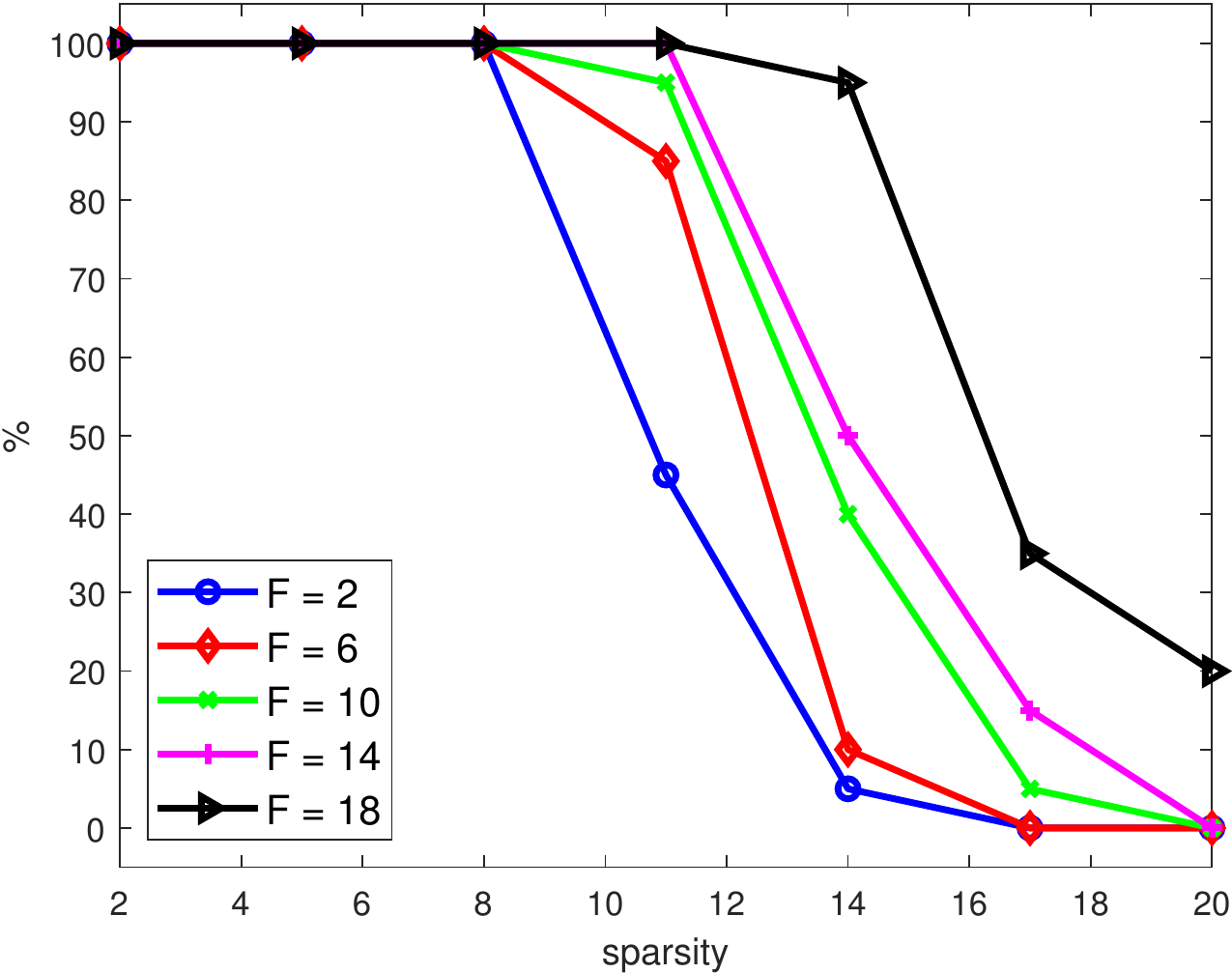}} 
	\subfloat[Success rates without MS]{\label{fig:SR_nonms}\includegraphics[width=0.45\textwidth]{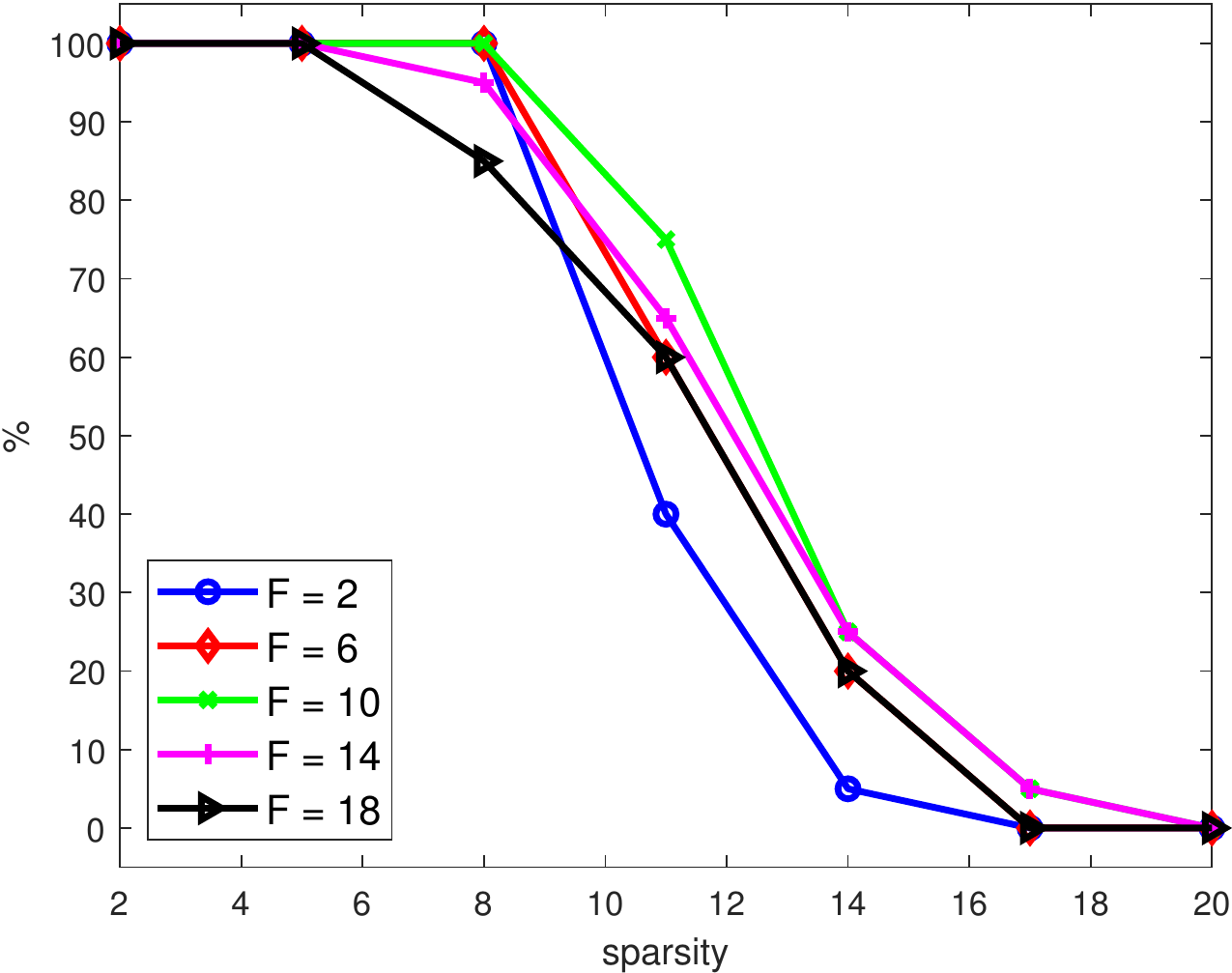}}
	\caption{The use of $\min\limits_{\h x}\left\{\frac{\|\h x\|_1}{\|\h x\|_2}: A\h x = \h 0 \right\}$ as an upper bound for  the $L_1$ recovery.  (a) plots the mean of ratios over 50 realizations with the standard deviation indicated as vertical bars. (b) and (c) are success rates of L1 recovery with and without minimum separation. }
	\label{fig:SR_ms_nonms}
\end{figure}

\section{Conclusions and future works}\label{sect:conclusion}

In this paper, we have studied a novel  $L_1/L_2$ minimization to promote sparsity. Two main benefits of $L_1/L_2$ are scale invariant and parameter free. 
Two numerical algorithms based on the ADMM are formulated for the assumptions of sparse signals and sparse gradients, together  with a variant of incorporating additional box constraint. The experimental results demonstrate the performance of the proposed approaches in comparison to the state-of-the-art methods in sparse recovery and MRI reconstruction. As a by-product, minimizing the ratio also gives an empirical upper bound towards $L_1$'s exact recovery, which motivates further investigations on exact recovery theories. Other  future works include algorithmic improvement and convergence analysis. In particular, it is shown in \Cref{tab:compare_time}, \Cref{fig:SR_MF_AF_uniform,fig:SR_MF_AF_dynamic,} that $L_1/L_2$ is not as fast as competing methods in CS and also  has certain algorithmic failures, which calls for a more robust and more efficient algorithm. In addition, we have provided  heuristic evidence of the ADMM's convergence in \Cref{fig:distance_auxiliary}  and it will be interesting to analyze it theoretically. 

\section*{Acknowledgements}
We would like to thank the editor and two reviewers for careful and thoughtful comments, which helped us greatly improve our paper. We also acknowledge the help of Dr.~Min Tao from Nanjing University, who suggested the reference on the ADMM convergence.

\section*{Appendix: proof of \Cref{thm:sNSP}}\label{sec:appendix}

In order to prove \Cref{thm:sNSP},
we study  the function 
	\begin{equation}
	g(t) =  \frac{\|\h x+t\h v\|_1^2}{\|\h x+t\h v\|_2^2}, 
	\end{equation}
	where $\h x \neq \h 0$\footnote{We assume that $\h b \neq 0 $ so $\h x=\h 0$ is not a solution to $A \h x = \h b $.} and 
	\begin{equation}\label{equ:v}
		\h v \in \ker(A)\backslash\h \{ \h 0\} \text{ with } \|\h v\|_2 = 1.
	\end{equation}	
Notice that the denominator of the function $ g$ is non-zero for all $t \in \mathbb{R} $. Otherwise, we have $\h x + t \h v = \h 0 $ and hence $A \h x + A (t \h v)  = A( \h 0)$. Since $ A \h x = \h b$ and $A \h v = \h 0,$ we get  $\h b = \h 0$ which is a contradiction. Therefore, the function $ g$ is continuous everywhere.	
Next, we introduce the following lemma to discuss the $L_1$ term in the numerator of $g$.  
	\begin{lemma}\label{lemma_x+tv}
	For any $\h  x\in \mathbb{R}^n$  and $\h  v \in \mathbb{R}^n $ satisfying \eqref{equ:v}, denote $S$ as the support of $\h x$ and $t_0 :=  \min\limits_{i \in S} |x_i|$.  We have 
	\begin{equation}\label{eq:x_tvL1}
	\|\h x+t\h v\|_1 = \|\h x\|_1 +t \sigma_t(\h v)>0, \quad \forall |t|<t_0, 
\end{equation}
where \begin{equation}\label{equ:simga}
\sigma_{t}(\h v) =  \sum_{i \in S} v_i \mbox{sign}(x_i) + \mbox{sign}(t)\|\h v_{\bar{S}}\|_1. 
\end{equation}
	\end{lemma}
	\begin{proof}
	
	Since $\h x + t \h v \neq \h 0 $ for all $t \in \mathbb{R} $, we have $\|\h x+t\h v\|_1 > 0$.
	It follows from \eqref{equ:v} that $|v_i|\leq 1, \ \forall i$. Then we get $\mbox{sign}(x_i+tv_i ) = \mbox{sign}(x_i),\  \forall i \in S$, as $ |tv_i| < |x_i|$ for $|t| < t_0$.
	 Therefore, we have
		\begin{equation}
\begin{split}
	\|\h x+t\h v\|_1  &=\sum_{i \in S}|x_i + tv_i| + \sum_{i \notin S} |t| |v_i| \\ &= \sum_{i \in S} (x_i + tv_i)\mbox{sign}(x_i) + |t| \|\h v_{\bar{S}}\|_1
	\notag \\
	&= \sum_{i \in S} x_i \mbox{sign}(x_i) + t \sum_{i \in S} v_i \mbox{sign}(x_i) + |t|\|\h v_{\bar{S}}\|_1\\
	&= \|\h x\|_1 + t \sum_{i \in S} v_i \mbox{sign}(x_i)  + |t| \|\h v_{\bar{S}}\|_1\notag\\
	&= \|\h x\|_1 + t \left( \sum_{i \in S} v_i \mbox{sign}(x_i)  + \mbox{sign}(t)\|\h v_{\bar{S}}\|_1 \right), \\
%	& = \|\h x\|_1 +t \sigma_t(\h v),
\end{split}
\end{equation}
which implies \eqref{eq:x_tvL1} and hence  \Cref{lemma_x+tv} holds. 
	\end{proof}

\medskip
	Notice that  $ \sigma_{t}(\h v) $ only relies on the sign of $t$, i.e., it is  constant  for $t>0$ and $t<0$. Therefore, $g(t)$ is differentiable on $ 0 < t < t_0 $ and $ -t_0 < t < 0 $ (Note that when $t\neq 0$, $g$ is not differentiable at the points where $x_i + tv_i = 0 $). 
	Some simple calculations lead to the derivative of $g$ for $ 0 < t < t_0 $ and $ -t_0 < t < 0 $, 
	\begin{equation}\label{equ:de_g}
	\begin{split}
	& g^\prime(t)
	= \frac{d}{d t} \left( \frac{\left( \|\h x\|_1 + t \sigma_{t}(\h v) \right)^2}{\|\h x\|_2^2 +2t\left\langle \h v_S,\h x \right\rangle + t^2\|\h v\|_2^2} \right)
	\\
	&=  \frac{2\sigma_{t}(\h v) \left( \|\h x\|_1 + t \sigma_{t}(\h v) \right)\left( \|\h x\|_2^2 +2t\left\langle \h v_S,\h x \right\rangle + t^2\|\h v\|_2^2 \right) - \left(2\left\langle \h v_S,\h x \right\rangle + 2t\|\h v\|_2^2 \right)\left( \|\h x\|_1 + t \sigma_{t}(\h v) \right)^2}{\left( \|\h x\|_2^2 +2t\left\langle \h v_S,\h x \right\rangle + t^2\|\h v\|_2^2 \right)^2}
	\\
	&= \frac{2 \left( \|\h x\|_1 + t \sigma_{t}(\h v) \right)\left[ \sigma_{t}(\h v) \left( \|\h x\|_2^2 +2t\left\langle \h v_S,\h x \right\rangle + t^2\|\h v\|_2^2 \right) - \left(\left\langle \h v_S,\h x \right\rangle + t\|\h v\|_2^2 \right)\left( \|\h x\|_1 + t \sigma_{t}(\h v) \right) \right]}{\left( \|\h x\|_2^2 +2t\left\langle \h v_S,\h x \right\rangle + t^2\|\h v\|_2^2 \right)^2}
	\\
	&= \frac{2 \left( \|\h x\|_1 + t \sigma_{t}(\h v) \right)\left[ \left(\sigma_{t}(\h v)\|\h x\|_2^2 - \left\langle \h v_S,\h x \right\rangle \|\h x\|_1 \right) + \left( \sigma_{t}(\h v)\left\langle \h v_S,\h x \right\rangle - \|\h x\|_1\|\h v\|_2^2 \right)t \right]}{\left( \|\h x\|_2^2 +2t\left\langle \h v_S,\h x \right\rangle + t^2\|\h v\|_2^2 \right)^2}. 
	\end{split}
	\end{equation}
It follows from \Cref{lemma_x+tv} that the first term in the numerator of  \eqref{equ:de_g} is strictly positive, i.e., 	$\|\h x\|_1 +t \sigma_t(\h v)>0$. Therefore, the sign of $g^\prime$ depends on the second term in the numerator.  We further introduce two lemmas  (\Cref{lemma1} and \Cref{lemma2}) to study this term.   

\begin{lemma}\label{lemma1}
	For any  $\h  x,\h v \in \mathbb{R}^n $ and  $ i \in [n]$, we have
	\begin{eqnarray}
	&&	n\|\h x\|_2^2 - |x_i|\|\h x\|_1 \geq   (n-1)\left( \sum_{j \neq i} x_j^2 \right),\label{ineq:lem1x}\\
	&& n	\|\h v\|_1\|\h x\|_2^2 \geq \|\h x\|_1 |\left\langle \h v,\h x \right\rangle|.\label{ineq:lem1xv}
	\end{eqnarray}
	Furthermore, if $\|\h x\|_0=s$, then the constant $n$ in the inequalities can be reduced to $s$.		
\end{lemma}

\begin{proof}
	Simple calculations show that 
	\begin{equation}
	\begin{split}
	n\|\h x\|_2^2 - |x_i|\|\h x\|_1 &= n\left( \sum_{j} x_j^2 \right) - |x_i|\left( \sum_{j} |x_j| \right)
	%	\\ &= 
	%	(n-1)\left( \sum_{j \neq i} x_j^2 \right)  +   \sum_{j \neq i}x_j^2 + nx_i^2  -  \sum_{j}|x_i||x_j|
	\\ &= 
	(n-1)\left( \sum_{j \neq i} x_j^2 \right)  +   \sum_{j \neq i}x_j^2 + (n-1)x_i^2  -  \sum_{j \neq i}|x_i||x_j|
	%	\\ &= 
	%	(n-1)\left( \sum_{j \neq i} x_j^2 \right)+   \sum_{j \neq i}\left(x_j^2 - 2|x_i||x_j| +x_i^2 + |x_i||x_j| \right)
	\\ &=
	(n-1)\left( \sum_{j \neq i} x_j^2 \right)+   \sum_{j \neq i}\left( (|x_i| - |x_j|)^2 + |x_i||x_j| \right)
	\\ &
	\geq     (n-1)\left( \sum_{j \neq i} x_j^2 \right) \geq 0.
	\end{split}
	\end{equation}
	Therefore, we have
	$
	\sum_{i} \left( n\|\h x\|_2^2 - |x_i|\|\h x\|_1 \right) |v_i|  \geq 0,
	$
	which implies that 
	\begin{equation}
	n\|\h v\|_1\|\h x\|_2^2 \geq \|\h x\|_1\left( \sum_{i} |x_i||v_i| \right)  \geq \|\h x\|_1	|\left\langle \h v,\h x \right\rangle|.
	\end{equation}
		Similarly, we can reduce the constant $n$ to $s$,  if we know  $\|\h x\|_0 = s$. 
\end{proof}

\begin{lemma}\label{lemma2}		
	Suppose that an $s$-sparse vector $\h x$ satisfies $A\h x =\h b \ (\h b \neq 0)$ with its support on an index set $S$ and the matrix $A$ satisfies  the sNSP of order $ s$.  Define
	\begin{equation}
	t_1:= \inf_{\h v,t} \left\lbrace \frac{ |\sigma_{t}(\h v)\|\h x\|_2^2 - \left\langle \h v_S,\h x \right\rangle \|\h x\|_1 |}{ | \sigma_{t}(\h v)\left\langle \h v_S,\h x \right\rangle - \|\h x\|_1\|\h v\|_2^2|} \, \middle| \,\h v \in \ker(A), \|\h v\|_2 = 1, t \neq 0  \right\rbrace,
	\end{equation}
	where 	$\sigma_{t}(\h v)$ is defined as \eqref{equ:simga}.  
	 Then $t_1>0.$
\end{lemma}

\begin{proof}
	For any $\h v \in \ker(A) $ and $\|\h v\|_2 = 1$, it is straightforward that
	\begin{equation}
	\begin{split}
	| \sigma_{t}(\h v)\left\langle\h v_S,\h x \right\rangle - \|\h x\|_1\|\h v\|_2^2| & \leq | \sigma_{t}(\h v)|\|\h v\|_2\|\h x\|_2 + \|\h x\|_1\|\h v\|_2^2
	\\ & \leq 
	\|\h v\|_1\|\h v\|_2\|\h x\|_2 + \|\h x\|_1\|\h v\|_2^2 \\
	& = 
	\|\h v\|_1\|\h x\|_2 + \|\h x\|_1 \\
	& \leq 
	\sqrt{n}\|\h x\|_2 + \|\h x\|_1,
	\end{split}
	\end{equation}
	and
	\begin{equation}
	|\sigma_{t}(\h v)| \geq | \mbox{sign}(t)\|\h v_{\bar{S}}\|_1 |-|  \sum_{i \in S} v_i \mbox{sign}(x_i) | \geq 
	\|\h v_{\bar{S}}\|_1 - \sum_{i \in S} |v_i| = \|\h v_{\bar{S}}\|_1 - \|\h v_S\|_1.
	\end{equation}
	It follows from the sNSP that $ \|\h v_{\bar{S}}\|_1 \geq (s+1)\|\h v_S\|_1 $, thus leading to the following two inequalities, 	
	\begin{equation}\label{equ:sigma_ineq}
	\begin{split}
	|\sigma_{t}(\h v)| &\geq \|\h v_{\bar{S}}\|_1 - \|\h v_S\|_1 \geq s\|\h v_S\|_1 \\ 		|\sigma_{t}(\h v)| &\geq \|\h v_{\bar{S}}\|_1 - \|\h v_S\|_1 \geq (1 - \frac{1}{s+1})\|\h v_{\bar{S}}\|_1 =\frac{s}{s+1}\|\h v_{\bar{S}}\|_1.
	\end{split}
	\end{equation}	
	Next we will discuss two cases:  $s=1$ and $s>1$. 
	
	\begin{enumerate}	
		\item[(i)] For $s = 1 $. Without loss of generality, we assume the only non-zero element is $x_n \neq 0$ and hence 
		%$\sigma_{t}(v) =  v_i \mbox{sign}(x_i)  + \mbox{sign}(t)\|\h v_{\bar{S}}\|_1 $. Therefore, 
		we have
		\begin{equation*}
		|\sigma_{t}(\h v)\|\h x\|_2^2 - \left\langle \h v_S,\h x \right\rangle \|\h x\|_1 |  =   |( v_n \mbox{sign}(x_n)  + \mbox{sign}(t)\|\h v_{\bar{S}}\|_1)x_n^2 -(v_nx_n)|x_n| |=
		\|\h v_{\bar{S}}\|_1x_n^2.
		\end{equation*}
		We further discuss two cases:  $|v_n| \geq \frac{1}{\sqrt{n}}$ and  $|v_n| < \frac{1}{\sqrt{n}}$.
		If  $|v_n| \geq \frac{1}{\sqrt{n}}$, then $ \|\h v_{\bar{S}}\|_1 \geq (s+1)|v_n| \geq \frac{s+1}{\sqrt{n}} $ and  hence
		\begin{equation}\label{eq:lemma_s1_geg}
		|\sigma_{t}(\h v)\|\h x\|_2^2 - \left\langle \h v_S,\h x \right\rangle \|\h x\|_1 |  =   \|\h v_{\bar{S}}\|_1\|\h x\|_2^2 \geq \frac{s+1}{\sqrt{n}} \|\h x\|_2^2. 
		\end{equation}
		If $|v_n| < \frac{1}{\sqrt{n}}$, then we have $ \|\h v_{\bar{S}}\|_1 \geq 1 - |v_n| = 1 - \frac{1}{\sqrt{n}} = \frac{\sqrt{n} - 1}{\sqrt{n}} $ and
		\begin{equation}\label{eq:lemma_s1_leq}
		|\sigma_{t}(v)\|\h x\|_2^2 - \left\langle \h v_S,\h x \right\rangle \|\h x\|_1 |  =   \|\h v_{\bar{S}}\|_1\|\h x\|_2^2 \geq \frac{\sqrt{n} - 1}{\sqrt{n}} \|\h x\|_2^2. 
		\end{equation}
		Combining \eqref{eq:lemma_s1_geg} and \eqref{eq:lemma_s1_leq}, we have
		\begin{equation}\label{equ:t_1_part2}
		t_1 \geq \frac{ \min \left\lbrace \frac{s+1}{\sqrt{n}} \|\h x\|_2^2, \frac{\sqrt{n} - 1}{\sqrt{n}} \|\h x\|_2^2  \right\rbrace }{\sqrt{n}\|\h x\|_2 + \|\h x\|_1}>0. 
		\end{equation}

	\item[(ii)] 
		 For $s>1$. 
We split into two cases. The first case is $\forall j\in S, \ v_j < c $ (we will determine the value of $c$ shortly). As a result, we get $\|\h v_{S} \|_1< s c $  and $\|\h v_{\bar{S}}\|_1\geq 1-sc$ since $\|\h v\|_1 \geq  \|\h v\|_2 = 1$.  Some simple calculations lead to
	\begin{equation*}
	\begin{split}
	\left\vert\sigma_{t}(\h v)\|\h x\|_2^2 - \left\langle\h v_S,\h x \right\rangle \|\h x\|_1 \right\vert  & \geq  |\sigma_{t}(\h v)|\|\h x\|_2^2 - |\left\langle \h v_S,\h x \right\rangle |\|\h x\|_1  
	\\  & \geq 
	\frac{s}{s+1}\|\h v_{\bar{S}}\|_1\|\h x\|_2^2 - \sum_{i\in S}|v_i||x_i| \|x\|_1 \quad (\text{based on }  \eqref{equ:sigma_ineq})\\
	&\geq
	\frac{s}{s+1}(1-sc)\|\h x\|_2^2-\sum_{i\in S}c|x_i|\|\h x\|_1
	\\		& =
	\frac{s}{s+1}(1-sc)\|\h x\|_2^2-c\|\h x\|_1^2  \\
	& \geq \frac{s}{s+1}(1-su_0)\|\h x\|_2^2-sc\|\h x\|_2^2 \\
	& = \frac s {s+1}\big(1-(2s+1)c\big)\|\h x\|_2^2.
	\end{split}
	\end{equation*}
	If we choose $c=\frac 1 {2s+2},$ then the above quantity is larger than $\frac{s\|\h x\|_2 ^2}{(s+1)(2s+2)}>0.$

In the second case, we have there exist $j\in S$ such that  $ v_j \geq c$, leading to 
%\begin{equation}\label{equ:t_0}
%	c_0 :=  \min_{i\in S} |v_i| > 0
%\end{equation}
% Moreover, 
\begin{equation}
		\begin{split}
			\left\vert\sigma_{t}(\h v)\|\h x\|_2^2 - \left\langle \h v_S,\h x \right\rangle \|\h x\|_1 \right\vert  &\geq  |\sigma_{t}(\h v)|\|\h x\|_2^2 - |\left\langle \h v_S,\h x \right\rangle |\|\h x\|_1  
		%		\\ & \geq 
		%		s\|v_s\|_1\|x\|_2^2 - |\left\langle v_S,x \right\rangle |\|x\|_1 
		\\ &\geq 
		s\|\h v_S\|_1\|\h x\|_2^2 - \left(\sum_{i \in S}  |x_i||v_i| \right) \|\h x\|_1  \\
		&=
		\sum_{i \in S} \left( s\|\h x\|_2^2 - |x_i|\|\h x\|_1 \right)|v_i|
		\\ & \geq 
		\left( s\|\h x\|_2^2 - |x_j|\|\h x\|_1 \right)|v_j| \quad (\text{based on } \Cref{ineq:lem1x})\\
		& \geq
		c \left( s\|\h x\|_2^2 - |x_j|\|\h x\|_1 \right)
		\\ & \geq 
		c (s-1)\sum_{i \neq j}x_i^2 \quad (\text{based on } \Cref{ineq:lem1x}) \\
		& \geq
		c(s-1) \min_{j \in S} \sum_{i \neq j}x_i^2.
		\end{split}
		\end{equation}	
	These two cases guarantee that $t_1>0$, i.e.,
		\begin{equation}\label{equ:t_1_part1}
		t_1 \geq \frac{ \min \left\lbrace c(s-1) \min\limits_{j \in S} \sum\limits_{i \neq j}x_i^2,\ \frac{s\|\h x\|_2 ^2}{(s+1)(2s+2)} \right\rbrace }{\sqrt{n}\|\h x\|_2 + \|\h x\|_1}>0. 
		\end{equation}
	\end{enumerate}
	By \eqref{equ:t_1_part2} and \eqref{equ:t_1_part1}, we get \Cref{lemma2}. 
\end{proof}

\medskip
Now, we are ready to prove \Cref{thm:sNSP}. 
\begin{proof}	
	According to \eqref{eq:x_tvL1}, the first term in the numerator is strictly positive, i.e., 
	$ \|\h x\|_1 + t \sigma_{t}(\h v) = \|\h x + t\h v\|_1 > 0, \ \forall |t| < t_0$. As for the second one,  there exists a positive number $t_1$ defined in \Cref{lemma2} such that  
\begin{equation*}
	\left|\sigma_{t}(\h v)\left\langle \h v_S,\h x \right\rangle - \|\h x\|_1\|\h v\|_2^2 \right| |t| <|\sigma_{t}(\h v)\|\h x\|_2^2 - \left\langle \h v_S,\h x \right\rangle \|\h x\|_1 |
\end{equation*}
	for all $|t| < t_1$ and  $\h v \in \ker(A) \text{ with } \|\h v\|_2 = 1. $ Moreover, we have
	\begin{eqnarray*}
		&&	\mbox{sign} \Big[ (\sigma_{t}(\h v)\|\h x\|_2^2 - \left\langle \h v_S,\h x \right\rangle \|\h x\|_1  + \left( \sigma_{t}(\h v)\left\langle \h v_S,\h x \right\rangle - \|\h x\|_1\|\h v\|_2^2 \right)t \Big]\\
		&  =& \mbox{sign} \Big(\sigma_{t}(\h v)\|\h x\|_2^2 - \left\langle \h v_S,\h x \right\rangle \|\h x\|_1 \Big).
	\end{eqnarray*}
	Letting $t^* = \min\{t_0,t_1\} $, we have for any $t\in (0,t^*)$ and $\h v\neq \h 0$ that  $\sigma_{t}(\h v) > 0 $ as
	\begin{equation}
	\begin{split}
	\sigma_{t}(\h v) &=  \sum_{i \in S} v_i \mbox{sign}(x_i)  + \mbox{sign}(t)\|\h v_{\bar{S}}\|_1\\
	&= 
	 \sum_{i \in S} v_i \mbox{sign}(x_i) + \|\h v_{\bar{S}}\|_1
	\\ & \geq 
	\|\h v_{\bar{S}}\|_1 -  \|\h v_S\|_1 \\
	& \geq 
	\max \left\lbrace s\|\h v_S\|_1, \frac{s}{s+1}  \| 
	\h v_{\bar{S}}\|_1  \right\rbrace > 0  \quad (\text{based on } \eqref{equ:sigma_ineq}). 
	\end{split}
	\end{equation}
	 Also \eqref{equ:sigma_ineq} implies that
	\begin{equation}
	|\sigma_{t}(\h v)|\|\h x\|_2^2 \geq s \|\h v_S\|_1 \|\h x \|_2^2 \geq \|\h v_S\|_1 \|\h x \|_1^2 \geq |\left\langle \h v_S,\h x \right\rangle| \|\h x\|_1 ,
	\end{equation}
	thus leading to
	\begin{equation}
	\sigma_{t}(\h v)\|\h x\|_2^2 - \left\langle \h v_S,\h x \right\rangle \|\h x\|_1 \geq 0,
	\end{equation}
	for $\sigma_{t}(\h v) > 0.$ 
	As a result,   we have $g^{\prime}(t) \geq 0 $ if $0 < t < t^* $. The function $g(t)$  is not differentiable at zero, but  we can compute the sub-derivative as follows,
	\begin{equation}
	g^{\prime}(0^+) = \lim_{t \to  0^+} \frac{g(t) - g(0)}{t - 0} = \frac{2 \|\h x\|_1 \left(\sigma_{+1}(\h v)\|\h x\|_2^2 - \left\langle \h v_S,\h x \right\rangle \|\h x\|_1 \right)}{ \|\h x\|_2^4} \geq 0.
	\end{equation}
	Similarly, we can get $g^{\prime}(t) \leq 0 $ if $-t^* < t < 0 $ and $g^{\prime}(0^-)\leq 0$. Therefore for any $ 0<|t| < t^*  $ we have $g(0) \leq g(t) $, which implies that 
	\begin{equation}
	\frac{\|\h x+t\h v\|_1}{\|\h x+t\h v\|_2} \geq  \frac{\|\h x\|_1}{\|\h x\|_2}, \qquad \forall |t| < t^*.
	\end{equation}
	Notice that $t^* $ does not depend on the choice of $\h v $, therefore the inequality is true for any $\h v $ satisfying \ref{equ:v}, which will imply the result.
	
\end{proof}

\bibliographystyle{siamplain}
\bibliography{refer_l1dl2}

\begin{thebibliography}{10}

\bibitem{bandeira2013certifying}
{\sc A.~S. Bandeira, E.~Dobriban, D.~G. Mixon, and W.~F. Sawin}, {\em
  Certifying the restricted isometry property is hard}, IEEE Trans. Inf.
  Theory, 59 (2013), pp.~3448--3450.

\bibitem{beck2017first}
{\sc A.~Beck}, {\em First-Order Methods in Optimization}, vol.~25, SIAM, 2017.

\bibitem{bolte2014proximal}
{\sc J.~Bolte, S.~Sabach, and M.~Teboulle}, {\em Proximal alternating
  linearized minimization for nonconvex and nonsmooth problems}, Math.
  Program., 146 (2014), pp.~459--494.

\bibitem{boydPCPE11admm}
{\sc S.~Boyd, N.~Parikh, E.~Chu, B.~Peleato, and J.~Eckstein}, {\em Distributed
  optimization and statistical learning via the alternating direction method of
  multipliers}, Found. Trends Mach. Learn., 3 (2011), pp.~1--122.

\bibitem{candes2014towards}
{\sc E.~J. Cand{\`e}s and C.~Fernandez-Granda}, {\em Towards a mathematical
  theory of super-resolution}, Comm. Pure Appl. Math., 67 (2014), pp.~906--956.

\bibitem{CRT}
{\sc E.~J. Cand\`es, J.~Romberg, and T.~Tao}, {\em Stable signal recovery from
  incomplete and inaccurate measurements}, Comm. Pure Appl. Math., 59 (2006),
  pp.~1207--1223.

\bibitem{candes2008introduction}
{\sc E.~J. Cand{\`e}s and M.~B. Wakin}, {\em An introduction to compressive
  sampling}, IEEE Signal Process. Mag., 25 (2008), pp.~21--30.

\bibitem{chambolleP11}
{\sc A.~Chambolle and T.~Pock}, {\em A first-order primal-dual algorithm for
  convex problems with applications to imaging}, J. Math. Imaging and Vision,
  40 (2011), pp.~120--145.

\bibitem{chartrand07}
{\sc R.~Chartrand}, {\em Exact reconstruction of sparse signals via nonconvex
  minimization}, IEEE Signal Process. Lett., 10 (2007), pp.~707--710.

\bibitem{cohen2009compressed}
{\sc A.~Cohen, W.~Dahmen, and R.~DeVore}, {\em Compressed sensing and the best
  k-term approximation}, J. Am. Math. Soc., 22 (2009), pp.~211--231.

\bibitem{donohoE03}
{\sc D.~Donoho and M.~Elad}, {\em Optimally sparse representation in general
  (nonorthogonl) dictionaries via $l_1$ minimization}, Proc. Nat. Acad. Scien.
  USA, 100 (2003), pp.~2197--2202.

\bibitem{donoho2001uncertainty}
{\sc D.~L. Donoho and X.~Huo}, {\em Uncertainty principles and ideal atomic
  decomposition}, IEEE Trans. Inf. Theory, 47 (2001), pp.~2845--2862.

\bibitem{esserLX13}
{\sc E.~Esser, Y.~Lou, and J.~Xin}, {\em A method for finding structured sparse
  solutions to non-negative least squares problems with applications}, SIAM J.
  Imaging Sci., 6 (2013), pp.~2010--2046.

\bibitem{DCT2012coherence}
{\sc A.~Fannjiang and W.~Liao}, {\em Coherence pattern--guided compressive
  sensing with unresolved grids}, SIAM J. Imaging Sci., 5 (2012), pp.~179--202.

\bibitem{GO}
{\sc T.~Goldstein and S.~Osher}, {\em The split {B}regman method for
  ${L_1}$-regularized problems}, SIAM J. Imaging Sci., 2 (2009), pp.~323--343.

\bibitem{gribonval2003sparse}
{\sc R.~Gribonval and M.~Nielsen}, {\em Sparse representations in unions of
  bases}, IEEE Trans. Inf. Theory, 49 (2003), pp.~3320--3325.

\bibitem{guo2017convergence}
{\sc K.~Guo, D.~Han, and T.-T. Wu}, {\em Convergence of alternating direction
  method for minimizing sum of two nonconvex functions with linear
  constraints}, Int. J. of Comput. Math., 94 (2017), pp.~1653--1669.

\bibitem{cs:edgecs_guo2012}
{\sc W.~Guo and W.~Yin}, {\em Edge guided reconstruction for compressive
  imaging}, SIAM J. Sci. Imaging, 5 (2012), pp.~809--834.

\bibitem{hong2016convergence}
{\sc M.~Hong, Z.-Q. Luo, and M.~Razaviyayn}, {\em Convergence analysis of
  alternating direction method of multipliers for a family of nonconvex
  problems}, SIAM J. Optim., 26 (2016), pp.~337--364.

\bibitem{hoyer2002}
{\sc P.~O. Hoyer}, {\em Non-negative sparse coding}, in Proc. IEEE Workshop
  Neural Networks Signal Proce., 2002, pp.~557--565.

\bibitem{hurleyR09}
{\sc N.~Hurley and S.~Rickard}, {\em Comparing measures of sparsity}, IEEE
  Trans. on Inform. Theory, 55 (2009), pp.~4723--4741.

\bibitem{krishnan2011blind}
{\sc D.~Krishnan, T.~Tay, and R.~Fergus}, {\em Blind deconvolution using a
  normalized sparsity measure}, in IEEE Conference on Computer Vision and
  Pattern Recognition (CVPR), IEEE, 2011, pp.~233--240.

\bibitem{laiXY13}
{\sc M.~J. Lai, Y.~Xu, and W.~Yin}, {\em Improved iteratively reweighted least
  squares for unconstrained smoothed lq minimization}, SIAM J. Numer. Anal., 5
  (2013), pp.~927--957.

\bibitem{li2015global}
{\sc G.~Li and T.~K. Pong}, {\em Global convergence of splitting methods for
  nonconvex composite optimization}, SIAM J. Optim., 25 (2015), pp.~2434--2460.

\bibitem{louOX15}
{\sc Y.~Lou, S.~Osher, and J.~Xin}, {\em Computational aspects of
  ${L_1}$-${L_2}$ minimization for compressive sensing}, in Model. Comput. \&
  Optim. in Inf. Syst. \& Manage. Sci., Adv. Intel. Syst. Comput., vol.~359,
  2015, pp.~169--180.

\bibitem{louYHX14}
{\sc Y.~Lou, P.~Yin, Q.~He, and J.~Xin}, {\em Computing sparse representation
  in a highly coherent dictionary based on difference of ${L_1}$ and ${L_2}$},
  J. Sci. Comput., 64 (2015), pp.~178--196.

\bibitem{louZOX15}
{\sc Y.~Lou, T.~Zeng, S.~Osher, and J.~Xin}, {\em A weighted difference of
  anisotropic and isotropic total variation model for image processing}, SIAM
  J. Imaging Sci., 8 (2015), pp.~1798--1823.

\bibitem{sparseMRI}
{\sc M.~Lustig, D.~L. Donoho, and J.~M. Pauly}, {\em Sparse {MRI}: The
  application of compressed sensing for rapid {MR} imaging}, Magnet. Reson.
  Med., 58 (2007), pp.~1182--1195.

\bibitem{lv2009unified}
{\sc J.~Lv and Y.~Fan}, {\em A unified approach to model selection and sparse
  recovery using regularized least squares}, Ann. Appl. Stat.,  (2009),
  pp.~3498--3528.

\bibitem{maLH17}
{\sc T.~Ma, Y.~Lou, and T.~Huang}, {\em Truncated ${L_1}$-${L_2}$ models for
  sparse recovery and rank minimization}, SIAM J. Imaging Sci., 10 (2017),
  pp.~1346--1380.

\bibitem{natarajan95}
{\sc B.~K. Natarajan}, {\em Sparse approximate solutions to linear systems},
  SIAM J. Comput.,  (1995), pp.~227--234.

\bibitem{optimization2014inc}
{\sc G.~Optimization}, {\em {Gurobi} optimizer reference manual}, 2015.

\bibitem{pang2018decomposition}
{\sc J.-S. Pang and M.~Tao}, {\em Decomposition methods for computing
  directional stationary solutions of a class of nonsmooth nonconvex
  optimization problems}, SIAM J. Optim., 28 (2018), pp.~1640--1669.

\bibitem{raguet2013generalized}
{\sc H.~Raguet, J.~Fadili, and G.~Peyr{\'e}}, {\em A generalized
  forward-backward splitting}, SIAM J. Imaging Sci., 6 (2013), pp.~1199--1226.

\bibitem{repetti2015euclid}
{\sc A.~Repetti, M.~Q. Pham, L.~Duval, E.~Chouzenouxe, and J.-C. Pesquet}, {\em
  Euclid in a taxicab: Sparse blind deconvolution with smoothed $\ell_1/\ell_2$
  regularization}, IEEE Signal Process. Lett., 22 (2015), pp.~539--543.

\bibitem{rudinOF92}
{\sc L.~Rudin, S.~Osher, and E.~Fatemi}, {\em Nonlinear total variation based
  noise removal algorithms}, Physica D, 60 (1992), pp.~259--268.

\bibitem{shen2012likelihood}
{\sc X.~Shen, W.~Pan, and Y.~Zhu}, {\em Likelihood-based selection and sharp
  parameter estimation}, J. Am. Stat. Assoc., 107 (2012), pp.~223--232.

\bibitem{tillmann2014computational}
{\sc A.~M. Tillmann and M.~E. Pfetsch}, {\em The computational complexity of
  the restricted isometry property, the nullspace property, and related
  concepts in compressed sensing}, IEEE Trans. Inf. Theory, 60 (2014),
  pp.~1248--1259.

\bibitem{tran2017unified}
{\sc H.~Tran and C.~Webster}, {\em Unified sufficient conditions for uniform
  recovery of sparse signals via nonconvex minimizations}, arXiv preprint
  arXiv:1710.07348,  (2017).

\bibitem{wang2018convergences}
{\sc F.~Wang, W.~Cao, and Z.~Xu}, {\em Convergence of multi-block {B}regman
  {ADMM} for nonconvex composite problems}, Sci. China Info. Sci., 61 (2018),
  pp.~122101:1--12.

\bibitem{wang2014convergence}
{\sc F.~Wang, Z.~Xu, and H.-K. Xu}, {\em Convergence of {Bregman} alternating
  direction method with multipliers for nonconvex composite problems}, arXiv
  preprint arXiv:1410.8625,  (2014).

\bibitem{wang2019global}
{\sc Y.~Wang, W.~Yin, and J.~Zeng}, {\em Global convergence of {ADMM} in
  nonconvex nonsmooth optimization}, J. Sci. Comput., 78 (2019), pp.~29--63.

\bibitem{xuCXZ12}
{\sc Z.~Xu, X.~Chang, F.~Xu, and H.~Zhang}, {\em $l_{1/2}$ regularization: A
  thresholding representation theory and a fast solver}, IEEE Trans. Neural
  Networks, 23 (2012), pp.~1013--1027.

\bibitem{yinEX14}
{\sc P.~Yin, E.~Esser, and J.~Xin}, {\em Ratio and difference of $l_1$ and
  $l_2$ norms and sparse representation with coherent dictionaries}, Comm.
  Info. Systems, 14 (2014), pp.~87--109.

\bibitem{yinLHX14}
{\sc P.~Yin, Y.~Lou, Q.~He, and J.~Xin}, {\em Minimization of $\ell_{1-2}$ for
  compressed sensing}, SIAM J. Sci. Comput., 37 (2015), pp.~A536--A563.

\bibitem{zhangX17}
{\sc S.~Zhang and J.~Xin}, {\em Minimization of transformed $l_1$ penalty:
  Closed form representation and iterative thresholding algorithms}, Comm.
  Math. Sci., 15 (2017), pp.~511--537.

\bibitem{zhangX18}
{\sc S.~Zhang and J.~Xin}, {\em Minimization of transformed $l_1$ penalty:
  Theory, difference of convex function algorithm, and robust application in
  compressed sensing}, Math. Program., Ser. B, 169 (2018), pp.~307--336.

\bibitem{zhang2009multi}
{\sc T.~Zhang}, {\em Multi-stage convex relaxation for learning with sparse
  regularization}, in Adv. Neural. Inf. Process. Syst., 2009, pp.~1929--1936.

\bibitem{zhang2013theory}
{\sc Y.~Zhang}, {\em Theory of compressive sensing via {L1}-minimization: a
  non-{RIP} analysis and extensions}, J. Oper. Res. Soc. China, 1 (2013),
  pp.~79--105.

\end{thebibliography}

\end{document}